\documentclass[a4paper,11pt]{amsart}

\usepackage{amsmath, amssymb, amsthm}
\usepackage{graphicx}
\usepackage[top=3cm, bottom=3cm, left=2.5cm, right=2.5cm]{geometry}
\usepackage[colorlinks=true, linkcolor=blue, citecolor = blue]{hyperref}
\usepackage{mathrsfs}
\usepackage{wasysym}
\usepackage{color}

\newtheorem{thm}{Theorem}[section]
\newtheorem{lem}[thm]{Lemma}

\newtheorem{prop}[thm]{Proposition}
\newtheorem{defn}[thm]{Definition}
\newtheorem{assumpt}{Assumption}[section]

\newtheorem{rmk}{Remark}[section]

\newcommand{\Aavg}{\mathbb{A}}

\newcommand{\eps}{\varepsilon}

\newcommand{\om}{\omega}

\newcommand{\expt}{\mathbb{E}}

\newcommand{\R}{\mathbb{R}}
\newcommand{\gen}{\mathcal{A}}
\newcommand{\dom}{\mathcal{D}}
\newcommand{\C}{\mathcal{C}}
\newcommand{\Ch}{\hat{\C}}

\newcommand{\Th}{\hat{T}}

\newcommand{\Ind}{\mathbf{1}_{\{0\}}}
\newcommand{\la}{\langle}
\newcommand{\ra}{\rangle}
\newcommand{\pih}{\hat{\pi}}

\newcommand{\Psiz}{\tilde{\Psi}}

\newcommand{\igen}{\mathcal{L}}

\newcommand{\mbbP}{\mathbb{P}}

\newcommand{\Om}{\Omega}

\newcommand{\pj}{\heartsuit}
\newcommand{\icond}{\xi}
\newcommand{\F}{\mathcal{F}}
\newcommand{\bcfonC}{\clubsuit}
\newcommand{\qtf}{\clubsuit}
\newcommand{\qtfr}{\spadesuit}
\newcommand{\genonqtf}{A^0}

\newcommand{\ham}{\mathfrak{h}}
\newcommand{\htf}{h}
\newcommand{\hlev}{\hbar}
\newcommand{\stopt}{\mathfrak{e}}
\newcommand{\timepr}{\mathfrak{T}}

\newcommand{\sgenonqtf}{\mathcal{L}}
\newcommand{\tgenonqtf}{\mathfrak{B}}

\newcommand{\Ssp}{\mathbf{S}}
\newcommand{\eqvpj}{\varpi}
\newcommand{\tmprval}{2\pi/\omega_c}
\newcommand{\tmpr}{\mathfrak{T}}
\newcommand{\eqvtoH}{\sqsupset}
\newcommand{\igenH}{\mathcal{L}_H}
\newcommand{\poiF}{\mathfrak{F}}
\newcommand{\molli}{\mathfrak{m}}

\newcommand{\consa}{\kappa}
\newcommand{\hprc}{\mathcal{H}}
\newcommand{\expteps}{\expt}

\newcommand{\tauqext}{\widetilde{\tau}_q}
\newcommand{\genonqtfext}{\widetilde{\genonqtf}}
\newcommand{\avgHproc}{\check{h}}
\newcommand{\stoptH}{\check{\stopt}}

\newcommand{\Xeps}{X^\eps}

\newcommand{\Gqfun}{\varphi}
\newcommand{\Gqfunwh}{\widehat{\Gqfun}}

\newcommand{\stabsol}{h}

\newcommand{\modcont}{\mathfrak{w}}

\title[Perturbations of critical DDE: averaging approach]{Nonlinear and additive white noise perturbations of linear delay differential equations at the verge of instability: an averaging approach}
\author{N. Lingala and N. Sri Namachchivaya}
\address{University of Illinois,\\
Urbana, IL, USA.}
\email{lingala1@illinois.edu, navam@illinois.edu}
\keywords{Delay differential equation; averaging; martingale problem.}
\subjclass[2010]{34K06, 34K27, 34K33, 34K50}

\begin{document}

\begin{abstract}
The characteristic equation for a linear delay differential equation (DDE) has countably infinite roots on the complex plane. We deal with linear DDEs that are on the verge of instability, i.e. a pair of roots of the characteristic equation (critical eigenvalues) lie on the imaginary axis of the complex plane, and all other roots (stable eigenvalues) have negative real parts. We show that, when the system is perturbed by small noise, under an appropriate change of time scale, the law of the amplitude of projection onto the critical eigenspace is close to the law of a certain one-dimensional stochastic differential equation (SDE) without delay. Further, we show that the projection onto the stable eigenspace is small. These results allow us to give an approximate description of the delay-system using an SDE (without delay) of just one dimension. The proof is based on the martingale problem technique.
\end{abstract}

\maketitle


\section{Introduction}\label{sec:intro}
Delay differential equations (DDE) arise in a variety of areas such as manufacturing systems, biological systems, and control systems.  
In some of these systems, variation of a parameter would result in loss of stability through Hopf bifurcation---for example, see \cite{Gabor1} for machining processes and  \cite{Longtin} for the study of eye-pupil response to incident light. Typically these systems are also influenced by noise, for example, inhomogenity in the material properties of workpiece in machining processes \cite{Buck_Kusk}, and unmodeled dynamics in biological systems. Therefore, it is important to study the effect of noise in the models of such systems.

Linear DDEs possess countably infinite modes, i.e. the characteristic equation has countably infinite roots on the complex plane.  In this paper we deal with linear DDEs that are on the verge of instability, i.e. a pair of roots of the characteristic equation (critical eigenvalues) lie on the imaginary axis of the complex plane, and all other roots (stable eigenvalues) have negative real parts. We show that, when the system is perturbed by small noise, under an appropriate change of time scale, the law of the amplitude of projection onto the critical eigenspace, is close to the law of a certain one-dimensional stochastic differential equation (SDE). Further, we show that the projection onto the stable eigenspace is small. These results allow us to give an approximate description of the delay-system using an SDE (\emph{without delay}) of just one dimension and provide rigorous framework for the multi-scale analysis done in \cite{Klosek_Kuske_multi_siam}. Our proof is based on the martingale problem technique, and closely follows \cite{NavamSowers}.

Let $(\Om,\F,\mbbP)$ be a complete probability space and $W=\{W(t)\}_{0 \leq t \leq T}$ be a real valued Wiener process defined on it. Suppose $\{\F_t\}_{0\leq t\leq T}$ is a family of increasing $\mbbP$-complete sub-$\sigma$-fields of $\F$ such that 
$$\F_t^W \subset \F_t \text{ and } \F_t \perp \sigma\{W(v)-W(u),\,\,t\leq u\leq v\leq T\}.$$

Let $\C:=C([-r,0],\R)$. Furnished with $\sup$ norm, $\C$ is a Banach space. For any $f\in C([-r,T],\R)$, define the \emph{segment extractor} 
$$\pj_t f\,(\theta):=f(t+\theta), \qquad \quad -r\leq \theta \leq 0, \qquad 0\leq t\leq T.$$

Let $G:\C\to \R$ be a $C^2$ function satisfying the Lipshitz condition:
\begin{align}\label{eq:FLipcond}
|G(\eta)-G(\tilde{\eta})| & \leq K_G||\eta-\tilde{\eta}||, \quad \forall \,\eta,\tilde{\eta} \in \C,
\end{align}
and let $F:\C \to \R$ be defined by $F(\eta)=\sigma$ for all $\eta \in \C$ with $\sigma>0$.
It can be shown that there exists a constant $K_g$ such that $G$ satisfies the growth condition
\begin{align}\label{eq:FGgrowthforunique}
|F(\eta)|^2+|G(\eta)|^2\leq K_g(1+||\eta||^2), \qquad \forall \,\eta \in \C.
\end{align}

Our object of study is an $\R$-valued random process $X(t)$ satisfying
\begin{align}\label{eq:main_in_intgl_form}
X(t)=\begin{cases} \icond(0)+\int_0^tL_0(\pj_u X)du+\eps^2\int_0^t G(\pj_u X)du+\eps\int_0^t F(\pj_u X)d{W}(u), \qquad t\geq 0,  \\ \icond(t), \qquad \qquad  \qquad \qquad  \qquad \qquad  \qquad \qquad \qquad -r\leq t\leq 0,\end{cases} 
\end{align}
where $\xi$ is a  $\C$-valued square integrable random variable that is $\F_0$ measurable, $L_0:\C \to \R$ is a continuous linear mapping and $\eps<<1$. We write \eqref{eq:main_in_intgl_form} in short form as
$$dX(t)=L_0(\pj_t X)dt\,+\,\eps^2 G(\pj_t X)dt\,+\,\eps F(\pj_t X)dW(t), \qquad \pj_0X=\xi.$$

\begin{assumpt}\label{ass:assumptondetsys}
We assume that the corresponding deterministic DDE 
\begin{align}\label{eq:detDDE}
\dot{x}(t)=L_0(\pj_t x)
\end{align}
is critical, i.e. a pair of roots of the characteristic equation $\lambda-L_0e^{\lambda \cdot}=0$ are on the imaginary axis (critical eigenvalues) and all other roots have negative real parts (stable eigenvalues). 
\end{assumpt}

Using spectral theory, the space $\C$ can be split as $P_{\Lambda}\oplus Q_{\Lambda}$ where $P_{\Lambda}$ is determined solely by the critical eigenvalues. Denoting by $\pi \eta$ the projection of $\eta\in \C$ onto $P_{\Lambda}$, it can be shown that for the unperturbed system \eqref{eq:detDDE}, $||\pi \pj_t x||_{P_{\Lambda}}$ (see remark \ref{rmk:sayingHevolvesslow} for the norm) is a constant. When the system is perturbed by noise, as in \eqref{eq:main_in_intgl_form}, $\hprc(t):=\frac12 ||\pi \pj_t X||^2_{P_{\Lambda}}$  varies slowly.
We show that, as $\eps  \to 0$, the law of $\hprc(t/\eps^2)$ converges to the law of a one-dimensional SDE, and the projection of $\pj_t X$ onto $Q_{\Lambda}$ is small. For small $\eps$, the one-dimensional  system (\emph{without delay}) obtained  in the limit gives an approximate description of \eqref{eq:main_in_intgl_form}.

Reduced dimensional description of randomly perturbed conservative systems using the Hamiltonian is discussed, for example, in the works of Freidlin and Wentzell \cite{FWbook} and Namachchivaya and Sowers \cite{NavamSowers}, \cite{NavamSowersUnified}. Systems with random perturbations and fast decaying components are considered in \cite{PapKoh75}, \cite{vanRoes}. The current paper is an application of the above ideas for systems with delay. Naturally, the proofs presented here closely follow those in \cite{NavamSowers}.

This paper is organized as follows: Useful results on the unperturbed DDE \eqref{eq:detDDE} are collected in section \ref{sec:unpertprob}, and those on stochastic DDE are collected in section \ref{sec:SDDEframework}. The variation of constants formula, which expresses solution of \eqref{eq:main_in_intgl_form} using that of \eqref{eq:detDDE}, is discussed in section \ref{sec:vocform}. The noise perturbed system \eqref{eq:main_in_intgl_form} is considered in section \ref{sec:perturbedsys} where we perform a change of time-scale and show that the projection of solution onto $Q_{\Lambda}$ is small. In section \ref{sec:mainresult} we identify the generator for limiting dynamics of $\hprc(t/\eps^2)$ and state our main result, the proof of which is carried out in subsequent sections. An example is discussed in section \ref{sec:numsim}---the numerical simulations shown there illustrate the usefulness of the result. Though the equation \eqref{eq:main_in_intgl_form} that we consider is that of an $\R$-valued process, the theory holds for $\R^n$ valued processes also. However in the multidimensional case, it is easier to work with complexifications---and we discuss this in a separate article \cite{LingPREarxiv}.

\begin{rmk}\label{rmk:about_quad_perturb}
With a little extra effort, convergence of the law of $\hprc(t/\eps^2)$ may also be established for systems perturbed by slightly stronger deterministic perturbations:
\begin{align}\label{eq:quadnon_main_in_short_form}
dX(t)=L_0(\pj_t X)dt\,+\,\eps G_q(\pj_t X)dt\,+\,\eps^2 G(\pj_t X)dt\,+\,\eps F(\pj_t X)dW(t), 
\end{align}
where $G_q$ is such that a certain kind of time averaged effect of $G_q$ is zero. For example, $G_q(\eta)$ which are homogenously quadratic in $\eta$ (say $G_q(\eta)=(\eta(0))^2$) satisfy this property. This assumption is needed because otherwise the effect of $G_q$ is significant in just times of order $1/\eps$ whereas the effects of $G$ and $F$ are significant in times of order $1/\eps^2$. In the limit $\eps \to 0$, $G_q$ would result in two additional drift terms for the diffusion process limit of $\hprc(t/\eps^2)$. We defer this analysis to section \ref{sec:strongerdetperturb}.
\end{rmk}

\begin{rmk}\label{rmk:about_mult_perturb}
Most of the proof remains same even if we consider $F$ as a function of $\pj X$ instead of a constant $\sigma$. In appendix \ref{sec:multnoise} we consider this and show the necessary changes that need to be made to the proofs.
\end{rmk}

\section{The unperturbed deterministic system}\label{sec:unpertprob}

The content in this section is taken as it is from \cite{NavWihs2} which draws heavily from \cite{Halebook} and \cite{Diekmanbook}.

Let the space $\C:=C([-r,0],\R)$ be equipped with the sup-norm $||\phi|| = \sup_{-r \leq \theta \leq 0} |\phi(\theta)|$. We are interested in scalar DDE which are representible as linear autonomous retarded functional differential equation (RFDE) of the form
\begin{align}\label{eq:autoRFDE}
\dot{x}(t)&=L_0 \pj_t x, \qquad \quad t\geq 0, \\ \notag
\pj_0x&=\icond \in \C,
\end{align}
where $L_0:\C\to \R$ is a continuous linear mapping. The solution $x(t+\theta ; \icond)=\pj_t x(\theta;\icond)$ of the RFDE gives rise to the strongly continuous semigroup $T(t):\C\to \C, \, t\geq 0$,
$$(T(t)\icond)(\theta)=\pj_t x(\theta;\icond)$$ 
with generator $\gen$ given by 
\begin{equation}\label{eq:generator_def_unpert}
\gen \phi=\frac{d}{d\theta}\phi, \qquad dom(\gen)=\dom(\gen)=\{\phi \in \C^1|\phi'(0)=L_0\phi\}
\end{equation}
($\C^1$ is the linear space of continuously differentiable functions on $[-r, 0]$, and $'= \frac{d}{d\theta}$). With the initial condition $\icond$ in $\dom(\gen)$, the equation \eqref{eq:autoRFDE} is equivalent to the abstract differential equation
\begin{align}\label{eq:abstDE_unpert}
\frac{d}{dt}\pj_t x=\gen \pj_t x, \qquad t\geq 0, \quad \pj_0 x=\icond \in \dom(\gen), 
\end{align}
where the differentiation with respect to $t$ is taken in the sense of the sup-norm in $\C$.

\subsection{Spectral properties of $\dom(\gen)$ and decomposition of $\C$} The following lemma puts together known facts on the spectrum of $\gen$, $spec\gen$, pertinent to our study. On the basis of the spectrum we will decompose the space $\C$ into a two-dimensional subspace with maximal exponential growth rate and an infinite-dimensional space on which
all exponential growth rates are negative.
\begin{lem}\label{lem:unpert_spec}
Let $\gen$ defined by \eqref{eq:generator_def_unpert} be the generator of the semigroup $T(t)$ defined by the solution of the RFDE. Then
\begin{enumerate}
\item $\gen$ has only a point spectrum.
\item $\lambda \in spec\gen$ iff $\lambda$ satisfies the characteristic equation $\Delta(\lambda)=\lambda-L_0e^{\lambda \cdot}=0$.
\item For any real number $r$, $card\{\lambda \in spec\gen | Re(\lambda)>r\}<\infty$. (Re denotes the real part of).
\item For each eigenvalue $\lambda$ of $\gen$, both the generalized eigenspace $E_{\lambda}=Nul((\lambda I-\gen)^q)$ and the range $R_{\lambda}=Range((\lambda I-\gen)^q)$ are $\gen$-invariant and norm-closed linear subspaces of the complexification $\C_{\mathbb{C}}=E_{\lambda}\oplus R_{\lambda}$. (Here $I$ is the identity and $q$ is the algebraic multiplicity).
\end{enumerate}
\end{lem}

We make the following assumptions on $\gen$ (equivalently on $L_0$).
\begin{assumpt}
$\max\{Re(\lambda)\, | \,\lambda \in spec\gen \}=0$. The set of eigenvalues with maximum real part is $\Lambda=\{\lambda_1,\lambda_2\}=\{\pm i\omega_c\}$, $\omega_c>0$, where $i\omega_c$ satisfies the characteristic equation $\Delta(i\omega_c)=i\omega_c-L_0e^{i\omega_c \cdot}=0$.
\end{assumpt}
The corresponding eigenfunctions in $\C_{\mathbb{C}}$ are $\varphi_{1,2}=\Phi_1\pm i\Phi_2$ with $\Phi_1(\theta)=\cos(\omega_c\theta)$ and $\Phi_2(\theta)=\sin(\omega_c\theta),$ $\theta \in [-r,0]$. We introduce the row vector valued function
\begin{equation*}
\Phi(\cdot)=[\Phi_1(\cdot),\,\Phi_2(\cdot)]= [\cos(\omega_c\cdot),\,\sin(\omega_c\cdot)].
\end{equation*}
Using the identity $\cos(\omega_c(t+\cdot))=\cos(\omega_ct)\cos(\omega_c\cdot)-\sin(\omega_ct)\sin(\omega_c\cdot)$ and the linearity of $L_0$, it can be shown that
\begin{align}\label{eq:TPhi_eq_PhieB}
T(t)\Phi(\cdot)=\Phi(\cdot)e^{Bt}, \qquad B=\left[\begin{array}{cc}0 & \omega_c \\ -\omega_c & 0 \end{array}\right]
\end{align}
with the derivative
\begin{equation}\label{eq:APhi=PhiB}
\gen\Phi(\cdot)=\frac{d}{dt}T(t)\Phi(\cdot)|_{t=0}=\Phi(\cdot)B.
\end{equation}
These facts suggest the following decomposition of $\C$. Setting $E_{\Lambda}:=E_{\lambda_1}\oplus E_{\lambda_2}=span_{\mathbb{C}}\{\varphi_1,\,\varphi_2\}$, we obtain
\begin{lem}\label{lem:decomp_unpert}
\begin{enumerate}
\item The subspaces $P_{\Lambda}:=E_{\Lambda}\cap \C=span_{\R}\{\Phi_1,\Phi_2\}$ and $Q_{\Lambda}=(R_{\lambda_1}\oplus R_{\lambda_2})\cap \C$ are closed and $\gen$-invariant in $(\C,||\cdot||)$, and $\C=P_{\Lambda}\oplus Q_{\Lambda}.$
\item If $\pi$ denotes the projection of $\C$ onto $P_{\Lambda}$ along $Q_{\Lambda}$ (i.e. $\pi^2=\pi$ on $\C$ and $\pi(Q_{\Lambda})=0)$, then $ \,\pi(\dom(\gen))\subset P_{\Lambda} \subset \dom(\gen)$. 

$\gen$ is completely reducible w.r.t. $(P_{\Lambda},Q_{\Lambda})$ and $\pi(\gen(\phi))=\gen(\pi(\phi))$ for all $\phi \in \dom(\gen)$. The restrictions $\gen^{P}=\gen|_{P_{\Lambda}}$ and $\gen^Q=\gen|_{Q_{\Lambda}}$ satisfy $\overline{\dom(\gen^P)}=P_{\Lambda}$ and $\overline{\dom(\gen^Q)}=Q_{\Lambda}$ and generate a strongly continuous semigroup on $P_{\Lambda}$ and $Q_{\Lambda}$, respectively.
\item While $||T(t)|_{P_{\Lambda}}||_{op}=1$, there are positive constants $\kappa$ and $K$ such that for all $\phi \in Q_{\Lambda}$ (see \cite{Halebook}, page 215, Corollary 6.1),  
$$||T(t)\phi||\leq Ke^{-\kappa t}||\phi||, \qquad t\geq 0.$$
\end{enumerate}
\end{lem}
See Taylor and Lay \cite{Laybook}, Introduction to Functional Analysis, Theorem 12.5, page 248, and section V.5, pages 287-289, as well as Hale and Verduyn-Lunel \cite{Halebook}, chapter 6.

\subsubsection{Representation of the projection operator $\pi$.}\label{subsubsec:coordrepofpi} (\cite{Halebook}, pages 198, 212).
Define the bilinear form $\la \phi,\psi\ra$ on $C([-r,0],\R)\times C([0,r],\R)$, given by
\begin{equation}\label{eq:bilinform}
\la \phi,\psi \ra := \phi(0)\psi(0)-L_0(\int_0^{\cdot}\phi(u)\psi(u-\cdot)du)
\end{equation}
and introduce the column-vector valued function on $[0,r]$, $\Psi(\cdot)=\left[\begin{array}{c}\psi_1(\cdot) \\ \psi_2(\cdot) \end{array}\right]$, where $\psi_i$ are linear combinations of $\cos(\om \cdot)$ and $\sin(\om \cdot)$ and are such that $\la \Phi_i,\psi_j\ra=\delta_{ij}$. Putting $\la \phi, \Psi \ra = \left[\begin{array}{c}\la \phi, \psi_1 \ra \\ \la \phi, \psi_2\ra \end{array}\right]$, we obtain for the projection $\pi: \C \to P_{\Lambda}$,
\begin{align}\label{eq:projoper_intermsof_bilform}
\pi(\phi)=\Phi\la \phi,\Psi \ra = \la \phi, \psi_1\ra \Phi_1 + \la \phi, \psi_2\ra \Phi_2,
\end{align}
\begin{align}\label{eq:charac_of_QLambda}
Q_{\Lambda}=ker(\pi)=\{\phi \in \C | \pi(\phi)=0\}.
\end{align}

\subsubsection{Coordinate representation of $P_{\Lambda}$.} Identifying $P_{\Lambda}=\{\Phi z \,\,|\,\, z=\left[\begin{array}{c}z_1\\z_2\end{array}\right]\in \R^2\}$ with $\R^2$, putting $\pj_t x(\cdot)=\Phi(\cdot)z(t)+y_t(\cdot)$, where we define $y_t(\theta):=x(t+\theta)-\Phi(\theta)z(t)\in Q_{\Lambda}$, and taking into account \eqref{eq:APhi=PhiB}, we can replace the system \eqref{eq:abstDE_unpert}, i.e. $\frac{d}{dt}\pj_t x=\gen \pj_t x=\gen (\Phi(\cdot)z(t)+y_t(\cdot))$, by
\begin{equation}\label{eq:proj_eqns_unpert}
\dot{z}(t)=Bz(t), \qquad \quad \frac{d}{dt}y_t=\gen y_t
\end{equation} 
with initial values $z(0)$ and $y_0(\cdot)$ given by $\pj_0x(\cdot)=\Phi(\cdot)z(0)+y_0(\cdot)$.

\begin{rmk}\label{rmk:sayingHevolvesslow}
From the above equation, recalling the structure of $B$, one can see that $\frac12(z_1^2(t)+z_2^2(t))=\frac12\la \pj_t x,\Psi \ra^*\la \pj_t x,\Psi \ra$ is a constant. When we deal with the perturbed system \eqref{eq:main_in_intgl_form}, this quantity evolves much slowly compared to $X$. Let $\hprc(t):=\frac12\la \pj_tX,\Psi \ra^*\la \pj_t X,\Psi \ra$. Roughly, our aim is to show that the law of $\hprc(t/\eps^2)$ converges to that of a SDE without delay, whose generator would be specified later. 
\end{rmk}

\section{Stochastic DDE framework}\label{sec:SDDEframework}
Since the perturbed system \eqref{eq:main_in_intgl_form} is a stochastic DDE (SDDE) and we intend to use martingale problem technique to prove weak convergence of laws, here we collect the results of SDDE framework that would be useful to us.

\begin{thm}\label{thm:existuniqofsol2sdde2ndmom}{Theorem 1.3.1 in \cite{Kal_Mandal}, Theorem 2.2.1 in \cite{SEAM_book}}: Suppose that $(\Om,\F,\mbbP)$, $W$ and $\{\F_t\}$ are given as in section \ref{sec:intro}. Suppose $a,b:\C\to \R$ are two continuous functionals satisfying the Lipshitz condition
\begin{equation}\label{ass:Lip_coeff}
|a(\eta)-a(\tilde{\eta})|\,+\,|b(\eta)-b(\tilde{\eta})|\,\leq\,K||\eta-\tilde{\eta}||.
\end{equation}
Suppose $0\leq s\leq T$ and $\icond$ is a $\F_0$-measurable $\C$-valued random variable with $\expt\,||\icond||^2<\infty$. Then the SDDE with the initial process $\icond$, given by
\begin{align}\label{eq:main2_in_intgl_form}
X(t)=\begin{cases} \icond(0)+\int_0^ta(\pj_u X)du+\int_0^t b(\pj_u X)d{W}(u), \qquad t\geq 0,  \\ \icond(t), \qquad \qquad  \qquad \qquad  \qquad \qquad  \qquad \qquad \qquad -r\leq t\leq 0,\end{cases} 
\end{align}
possesses a unique continuous strong solution such that $X$ is $\F_t$ adapted and $\expt\,||\pj_t X||^2<\infty$.
\end{thm}

\begin{defn}\label{def:quasitamedefn}
(Definition IV.4.2 of \cite{SEAM_book}):
A function $\qtf:\C \to \R$ is said to be quasi-tame function if there exist an integer $k>0$, $C^\infty$-bounded maps $\qtfr:\R^k\to \R$, $f_j:\R\to \R$ and piecewise $C^1$ functions $g_j:[-r,0]\to \R$ for $j=1,\ldots,k-1$ such that
\begin{equation}\label{eq:qtf_form}
\qtf(\eta)=\qtfr([f_1,g_1|\eta],\ldots,[f_{k-1},g_{k-1}|\eta],\eta(0))
\end{equation}
for all $\eta \in \C$,
where $[f,g|\eta]=\int_{-r}^0 f(\eta(s))g(s)ds.$ The derivatives $g'_j$ are assumed to be absolutely integrable.

The space of quasi-tame functions is denoted by $\tau_q$.
\end{defn}

\begin{lem}\label{lem:howqtfevolves}
(Lemma 1.3.1 of \cite{Kal_Mandal}).
Suppose $X(t)$ is the solution to \eqref{eq:main2_in_intgl_form} and $f\in C_b(\R;\R)$ and $g\in C^1([-r,0];\R)$. Then
$$[f,g|\pj_tX]-[f,g|\pj_0 X]=\int_0^t\left\{ f(\pj_uX(0))g(0)-f(\pj_uX(-r))g(-r)-[f,g'|\pj_uX]\right\}du.$$
\end{lem}

\begin{lem}\label{lem:appliItoformula}
(pages 10-11 of \cite{Kal_Mandal}).
Let $\qtf \in \tau_q$ be of the form \eqref{eq:qtf_form} and $X(t)$ be the solution to \eqref{eq:main2_in_intgl_form}. Then 
\begin{align*}
\qtf(\pj_tX)-\qtf(\pj_0X)=&\sum_{j=1}^{k-1}\int_0^t\left\{ f_j(\pj_uX(0))g_j(0)-f_j(\pj_uX(-r))g_j(-r)-[f_j,g'_j|\pj_uX]\right\}\partial_j\qtfr\,du \\
&\quad + \int_0^t a(\pj_uX)\partial_k\qtfr\,du \,\,+\,\, \int_0^t b(\pj_uX)\partial_k\qtfr\,dW \,\,+\,\, \frac12 \int_0^t b^2(\pj_uX)\partial^2_k\qtfr\,du,
\end{align*}
where the partial derivatives are evaluated at $${([f_1,g_1|\pj_uX],\ldots,[f_{k-1},g_{k-1}|\pj_uX],\pj_uX(0))}.$$
\end{lem}

\begin{defn}\label{def:A0_def}
Denote by $C_b$ the Banach space of all bounded continuous functions $\bcfonC:\C \to \R$ with the sup norm
$$||\bcfonC||_{C_b}=\sup\{|\bcfonC(\phi)\,|\,:\,\phi \in \C\}.$$
Define an operator $\genonqtf$ on $C_b$ with $\dom(\genonqtf)=\tau_q$ as follows: Let $\qtf \in \dom(\genonqtf)$ be of the form \eqref{eq:qtf_form}. Then
\begin{align}\label{eq:A0_def}
(\genonqtf \qtf)(\eta)&:=\sum_{j=0}^{k-1}\bigg( f_j(\eta(0))g_j(0)-f_j(\eta(-r))g_j(-r)-[f_j,g_j'|\eta] \bigg)\partial_j\qtfr \notag \\ &\qquad \qquad + a(\eta) \partial_k\qtfr  \,\,+\,\, \frac12 b^2(\eta)\partial_k^2 \qtfr,
\end{align}
where the partial derivatives are evaluated at ${([f_1,g_1|\eta],\ldots,[f_{k-1},g_{k-1}|\eta],\eta(0))}$.
\end{defn}

\begin{thm}\label{thm:qtfmart_general}\footnote{Also see p.26 of \cite{SEAM_98_workshop} which works with additional assumption that the coefficients are globally bounded.} Theorem 1.3.2 of \cite{Kal_Mandal}: 
Suppose $X(t)$ for $-r\leq t\leq T$ is given by the SDDE \eqref{eq:main_in_intgl_form} with the coefficients $a,b$ satisfying the Lipshitz condition \eqref{ass:Lip_coeff}. Suppose $\qtf\in \tau_q$. Then
\begin{equation}\label{eq:qtf_mart}
M_t^{\qtf}:=\qtf(\pj_tX)-\qtf(\pj_0X)-\int_0^t (\genonqtf \qtf)(\pj_uX)du
\end{equation}
is a $\F_t$ martingale. 
\end{thm}


\section{The variation of constants formula}\label{sec:vocform}

In order to express the solution of \eqref{eq:main_in_intgl_form} using that of \eqref{eq:detDDE}, we need to use the solution of 
\begin{align}\label{eq:usefuleq_det_lin_unpert}
\dot{x}=L_0\pj_t x
\end{align}
with the initial condition $\pj_0\,x=\Ind$ where $\Ind$ is the indicator for $\{0\}$ over $[-r,0]$. Clearly $\Ind$ does not belong to $\C$ and so we need to extend the space $\C$.

The following lemma puts together the results pertaining to the extension. These are taken from p.192-193, 206-207 of \cite{SEAM_book} which makes use of \cite{Halebook}.

\begin{lem}\label{lem:collec_ext_res}
Let $\Ch:=\hat{C}([-r,0];\R)$ be the Banach space of all bounded measurable maps $[-r,0]\to \R$, given the sup norm.
\begin{enumerate}
\item Using Riesz representation it is possible to extend $L_0$ to $\Ch$. Denote this extension also by $L_0$. Solving the linear system \eqref{eq:usefuleq_det_lin_unpert}
for initial data in $\Ch$, the semigroup $T(t)$ can be extended to one on $\Ch$. Denote the extension by $\Th(t)$.
\item The representation \ref{subsubsec:coordrepofpi} of the projection operator $\pi$ gives a natural extension to a continuous linear map $\pih:\Ch \to P_{\Lambda}$. The formula \eqref{eq:projoper_intermsof_bilform} $\pih(\phi)=\Phi\la \phi,\Psi\ra$ holds even for $\phi \in \Ch$. In particular $\pih(\Ind)=\Phi \la \Ind, \Psi\ra =\Phi \Psi(0)$.
\item The space $\Ch$ has a topological splitting  $\Ch=P_{\Lambda}\oplus \hat{Q}_{\Lambda}$, where $$\hat{Q}_{\Lambda}=\{\eta \,:\, \eta \in \Ch,\,\, \pih \eta =0\}.$$
\item The above splitting is invariant under the semigroup $\Th$, i.e. for each $\eta \in \Ch$ and $t\geq 0$, we have
$$\pih \Th(t)\eta = \Th(t)\pih \eta, \qquad \quad (I-\pih)\Th(t)\eta = \Th(t)(I-\pih)\eta.$$
\item There exists positive constants $\kappa$ and $K$ such that for all $\phi \in \hat{Q}_{\Lambda}$ 
$$||\Th(t)\phi||\leq Ke^{-\kappa t}||\phi||, \qquad t\geq 0.$$
\end{enumerate}
\end{lem}

\begin{prop}\label{prop:vocformula}
Under the assumptions on $G$ and $L_0$ listed in section \ref{sec:intro}, the solution to 
\begin{align}\label{eq:main_in_intgl_form_for_VOC}
X(t)=\begin{cases} \icond(0)+\int_0^tL_0(\pj_u X)du+\int_0^t G(\pj_u X)du+\int_0^t \sigma d{W}(u), \qquad t\geq 0,  \\ \icond(t), \qquad \qquad  \qquad \qquad  \qquad \qquad  \qquad \qquad \qquad -r\leq t\leq 0,\end{cases} 
\end{align}
with $\xi \in \C$ such that $\expt||\xi||^2 < \infty$, satisfies the variation of constants formula
\begin{align}\label{eq:thevocform}
\pj_t X = \Th(t)\xi + \int_0^t \Th(t-u)\Ind G(\pj_u X)du + \int_0^t \Th(t-u)\Ind \sigma dW(u),
\end{align}
where
$$\bigg(\int_0^t \Th(t-u)\Ind \sigma dW(u)\bigg)(s):=\int_0^t  \big(\Th(t-u)\Ind \big)(s)\, \sigma dW(u), \qquad \forall s\in [-r,0],$$
$$\bigg(\int_0^t \Th(t-u)\Ind G(\pj_u X)du\bigg)(s):=\int_0^t  \big(\Th(t-u)\Ind \big)(s)\, G(\pj_u X)du, \,\,\, \forall s\in [-r,0].$$
Further the projections satisfy
\begin{align*}
\pi \pj_t X &= \Th(t)\pih\xi + \int_0^t \Th(t-u)\pih\Ind G(\pj_u X)du + \int_0^t \Th(t-u)\pih \Ind \sigma dW(u), \\
(I-\pi)\pj_t X &= \Th(t)(I-\pih)\xi + \int_0^t \Th(t-u)(I-\pih)\Ind G(\pj_u X)du \,\,\,+\\
& \qquad \qquad + \int_0^t \Th(t-u)(I-\pih) \Ind \sigma dW(u).
\end{align*}
\end{prop}
\begin{proof}
Similar to proof of theorem 4.1 on p.201 of \cite{SEAM_book}. It can be proved by Picard iteration that the equation
\begin{align}
y(t)=\begin{cases}(\Th(t)\xi)(0)+\int_0^t (\Th(t-u)\Ind)(0) G(\pj_u y)du +\int_0^t (\Th(t-u)\Ind)(0) \sigma dW(u),\\
 \hfill t\geq 0, \\ \xi(t), \qquad t\in[-r,0]\end{cases}
\end{align}
posseses a solution with continuous sample paths. Using arguments similar to the ones in p.201-202 of \cite{SEAM_book}, it can be shown that every such solution satisfies the FDE \eqref{eq:main_in_intgl_form_for_VOC}. And it is already known that solution to \eqref{eq:main_in_intgl_form_for_VOC} is unique. The formulas for the projections follow from the invariance of splitting $\Ch=P_{\Lambda}\oplus \hat{Q}_{\Lambda}$ under the semigroup $\Th$. 
\end{proof}


\section{The perturbed system}\label{sec:perturbedsys}
When the noise perturbation $\eps$ is small, the quantity $\hprc(t):=\frac12\la \pj_tX,\Psi \ra^*\la \pj_t X,\Psi \ra$ evolves slowly compared to $X$ (see remark \ref{rmk:sayingHevolvesslow}). Significant changes in $\hprc$ occur on time scales of order $1/\eps^2$. Hence we consider the perturbed system \eqref{eq:main_in_intgl_form} with a change of time scale: $\hat{X}(t):=X(t/\eps^2)$. For this purpose, define the segment extractor 
$$(\hat{\pj}^{\eps}_t f)\,(\theta)=f(t+\eps^2\theta), \qquad \quad -r\leq \theta \leq 0.$$
Then $\hat{X}(t)$ satisfies
\begin{align}\label{eq:main_in_intgl_form_timesclaechange}
\hat{X}(t)=\begin{cases} \icond(0)+\frac{1}{\eps^2}\int_0^tL_0(\hat{\pj}^{\eps}_u \hat{X})du+ \int_0^t G(\hat{\pj}^{\eps}_u \hat{X})du +\int_0^t F(\hat{\pj}^{\eps}_u \hat{X})d\hat{W}(u), \qquad t\geq 0,  \\ \icond(\eps^{-2} t), \qquad \qquad  \qquad \qquad  \qquad \qquad  \qquad \qquad \qquad -r\eps^2 \leq t\leq 0,\end{cases} 
\end{align}
where $\hat{W}(t)=\eps W(t/\eps^2)$.

We drop the hats on symbols and rewrite the notation:

Let $(\Om,\F,\mbbP)$ be a complete probability space and $W=\{W(t)\}_{0 \leq t \leq T}$ be a real valued Wiener process defined on it. Suppose $\{\F_t\}_{0\leq t\leq T}$ is a family of increasing $\mbbP$-complete sub-$\sigma$-fields of $\F$ such that 
$$\F_t^W \subset \F_t \text{ and } \F_t \perp \sigma\{W(v)-W(u),\,\,t\leq u\leq v\leq T\}.$$
We take the initial condition $\icond \in \C$ to be deterministic.

Our object of study is an $\R$-valued random process $\Xeps(t)$ satisfying
\begin{align}\label{eq:main_in_intgl_form_timesclaechange_drophat}
\Xeps(t)=\begin{cases} \icond(0)+\frac{1}{\eps^2}\int_0^tL_0(\hat{\pj}^{\eps}_u \Xeps)du+\int_0^t G(\hat{\pj}^{\eps}_u \Xeps)du+\int_0^t F(\hat{\pj}^{\eps}_u \Xeps)dW(u), \,\, t\geq 0,  \\ \icond(\eps^{-2} t), \qquad \qquad  \qquad \qquad  \qquad \qquad  \qquad \qquad \qquad -r\eps^2 \leq t\leq 0.\end{cases} 
\end{align}
We assume $G:\C\to \R$ is a $C^2$ function satisfying the Lipshitz condition \eqref{eq:FLipcond} and growth condition \eqref{eq:FGgrowthforunique}; $L_0$ is a continuous linear mapping that satisfies assumption \ref{ass:assumptondetsys}; and $F(\eta)=\sigma$ for all $\eta \in \C$.

Our aim is to study the weak convergence, as $\eps \to 0$, of the law of the scalar process $\hprc^{\eps}(t):=\ham(\hat{\pj}^{\eps}_t \Xeps)$ where\footnote{Note that $\Psi$ has two components. The symbol * indicates transpose of a vector.} $\ham(\eta):=\frac12\la \eta,\Psi \ra^*\la \eta,\Psi \ra$.

Note that $\ham$ solely depends on the $P_{\Lambda}$ projection. 
Proposition \ref{add:newprop:supboundGron} gives the justification to ignore the $Q_{\Lambda}$ projection. However proposition \ref{prop:stablenormtozero} about $Q_{\Lambda}$ projection is what we use in proving weak convergence of $\hprc^{\eps}(t)$.

\begin{defn}\label{def:stabsol}
Define $\stabsol:\R\to \R$, by
\begin{align*}
\stabsol(t)&=\begin{cases}(\Th(t)(I-\pih)\Ind)(0), \quad  t\geq 0, \\ \stabsol(0),  \quad t<0. \end{cases}
\end{align*}
Using part 5 of lemma \ref{lem:collec_ext_res} we have $|\stabsol(t)|\leq Ke^{-\kappa t}$ for $t\geq 0$. Further 
\begin{align}\label{eq:expdecayforhprime}
\left|\frac{d}{dt}\stabsol(t)\right|=\left|L_0\bigg(\Th(t)(I-\pih)\Ind\bigg)\right|\leq ||L_0||\,||\Th(t)(I-\pih)\Ind||\leq ||L_0||\,Ke^{-\kappa t}.
\end{align}
\end{defn}

Define the modulus of continuity for $f:[0,\infty)\to \R$
$$\modcont(a,b;f)=\sup_{\substack{|u-v|\,\leq\, a \\ u,v\,\in\, [0,b]}}|f(u)-f(v)|.$$

\begin{lem}\label{add:newlem:supbound}
Define $\Upsilon^{\eps}_s:=\left|\left|\int_0^s\Th(\frac{s-u}{\eps^2})(I-\pih)\Ind \sigma dW(u)\right|\right|$. Then there exists constants $\hat{C}$ and $\eps_{(a,\delta)}$ such that, given any $a\in [0,1)$, for $\eps < \eps_{(a,\delta)}$ 
\begin{align}\label{add:newlem:supbound_statement}
\mbbP[\sup_{s\in [0,T]}\Upsilon^{\eps}_s \,\geq \eps^a] \,\,\,\leq\,\,\,\hat{C}\eps^{-a}\sqrt{{r\eps^\delta}\ln\left(\frac{T}{r\eps^\delta}\right)}, \qquad \quad \delta \in (2a,2). 
\end{align}
\end{lem}
\begin{proof}
Let $\stabsol$ be as in definition \ref{def:stabsol}. Then
\begin{align}\label{add:newlem:supbound:eq:2_split1}
\sup_{s\in [0,T]}\Upsilon^{\eps}_s \,\,&\leq\,\,\sup_{s\in [0,T]}\sup_{\theta\in [-r,0]} \left|\int_0^{(s+\eps^2\theta)\vee 0}\stabsol\left(\frac{s+\eps^2\theta -u}{\eps^2}\right)\sigma dW(u)\right| \\ \label{add:newlem:supbound:eq:2_split2}
& \qquad + \sup_{s\in [0,T]}\sup_{\theta\in [-r,0]} \left|\int_{(s+\eps^2\theta)\vee 0}^s\pih \Ind\left(\frac{s+\eps^2\theta -u}{\eps^2}\right)\sigma dW(u)\right|
\end{align}
RHS of \eqref{add:newlem:supbound:eq:2_split1} equals
\begin{align}\label{add:newlem:supbound:eq:2_split1a}
\sup_{t\in [0,T]} \left|\int_0^{t}\stabsol\left(\frac{t -u}{\eps^2}\right)\sigma dW(u)\right| \,\,\leq & \,\,\sup_{t\in [r \eps^\delta,T]} \left|\int_0^{t-r \eps^\delta}\stabsol\left(\frac{t -u}{\eps^2}\right)\sigma dW(u)\right|  \\ \label{add:newlem:supbound:eq:2_split1b}
& \quad + \sup_{t\in [0,T]} \left|\int_{(t-r \eps^\delta)\vee 0}^t\stabsol\left(\frac{t -u}{\eps^2}\right)\sigma dW(u)\right| 
\end{align}
with $\delta \in (0,2)$.

Using integration by parts and exponential decay of $\stabsol$ and $\stabsol'$ in RHS of \eqref{add:newlem:supbound:eq:2_split1a} we have
\begin{align*}
RHS \,\,of\,\, \eqref{add:newlem:supbound:eq:2_split1a} \,\,\,\leq &\,\,\,\sup_{t\in [r \eps^\delta,T]}  |h(r \eps^{\delta-2})\sigma W(t-r \eps^{\delta})|  + \sup_{t\in [r \eps^\delta,T]} \frac{1}{\eps^2} \int_0^{t-r\eps^\delta}\left|\stabsol'\left(\frac{t -u}{\eps^2}\right)\right|\,|\sigma W(u)|\,du \\
\leq & \,\,\, K e^{-\kappa r \eps^{\delta-2}}\sup_{t\in [0,T]}|\sigma W(t)|\,\,+\,\, \sup_{t\in [r \eps^\delta,T]} \frac{1}{\eps^2}||L_0||K\int_0^{t-r\eps^\delta}e^{-\kappa (t-u)/\eps^2}\,|\sigma W(u)|\,du \\
\leq & \,\,\, K e^{-\kappa r \eps^{\delta-2}}\sup_{t\in [0,T]}|\sigma W(t)|\,\,+\,\, \frac{1}{\eps^2}||L_0||Ke^{-\kappa r \eps^{\delta-2}} \int_0^{T}|\sigma W(u)|\,du.
\end{align*}
For RHS of \eqref{add:newlem:supbound:eq:2_split1b}, we use
\begin{align*}
\stabsol\left(\frac{t -u}{\eps^2}\right)=\stabsol\left(\frac{t -((t-r\eps^{\delta})\vee 0)}{\eps^2}\right)-\frac{1}{\eps^2}\int_{(t-r\eps^{\delta})\vee 0}^u \stabsol'\left(\frac{t -\tau}{\eps^2}\right)d\tau, \qquad \text{ for } u\in [(t-r\eps^{\delta})\vee 0,t].
\end{align*}
Now,
\begin{align*}
RHS\,\,of\,\,\eqref{add:newlem:supbound:eq:2_split1b}\,\,&\leq\,\,
\sup_{t\in[0,r\eps^\delta]}|\stabsol(t \eps^{-2})|\,|\sigma W_t| \,+\, \sup_{t\in[r\eps^\delta,T]}|\stabsol(r\eps^{\delta-2})|\,|\sigma(W_t-W_{t-r\eps^\delta})| \\
& \qquad + \sup_{t\in [0,T]}\frac{1}{\eps^2}\left|\int_{(t-r \eps^\delta)\vee 0}^t\left(\int_{(t-r \eps^\delta)\vee 0}^u\stabsol'\left(\frac{t -\tau}{\eps^2}\right)d\tau\right)\sigma dW(u)\right| \\
&\leq 2K\sigma \modcont(r\eps^{\delta},T;W) + \sup_{t\in [0,T]}\frac{1}{\eps^2}\left|\int_{(t-r \eps^\delta)\vee 0}^t\left(\int_{\tau}^t \sigma dW(u)\right) \stabsol'\left(\frac{t -\tau}{\eps^2}\right)d\tau\right| \\
& \leq 2K\sigma \modcont(r\eps^{\delta},T;W) + \sup_{t\in [0,T]}\frac{1}{\eps^2}\int_{(t-r \eps^\delta)\vee 0}^t \sigma|W_t-W_{\tau}| \left|\stabsol'\left(\frac{t -\tau}{\eps^2}\right)\right|d\tau \\
& \leq 2K\sigma \modcont(r\eps^{\delta},T;W) + \sigma \modcont(r\eps^{\delta},T;W) \sup_{t\in [0,T]}\frac{1}{\eps^2}\int_{(t-r \eps^\delta)\vee 0}^t K||L_0||e^{-\kappa(t-\tau)/\eps^2}d\tau \\
& \leq 2K\sigma \left(1+\frac{||L_0||}{2\kappa}\right)\modcont(r\eps^\delta,T;W).
\end{align*}

For the term in \eqref{add:newlem:supbound:eq:2_split2} we make use of the following facts: 
\begin{align*}
\pih\Ind(v-u)&=\Psiz_1 \cos(v-u)+\Psiz_2\sin(v-u) \\
&=(\Psiz_1\cos v+\Psiz_2\sin v)\cos u + (\Psiz_1\sin v-\Psiz_2\cos v)\sin u, 
\end{align*}
and $|\Psiz_1\cos v+\Psiz_2\sin v|\leq \sqrt{\Psiz_1^2+\Psiz_2^2}=\sqrt{\Psiz^*\Psiz}$. Using these it is easy to see that the term in \eqref{add:newlem:supbound:eq:2_split2} is bounded above by  
\begin{align}\label{add:newlem:supbound:eq:2_split2_bound}
\sigma \sqrt{\Psiz^*\Psiz}\,\sup_{s\in [0,T]}\sup_{\theta\in [-r,0]} \left(\left|M^{c,\eps}_{s}-M^{c,\eps}_{(s+\eps^2\theta)\vee 0}\right| + \left|M^{s,\eps}_{s}-M^{s,\eps}_{(s+\eps^2\theta)\vee 0}\right|\right)
\end{align}
where 
\begin{align*}
M^{c,\eps}_t=\int_0^t \cos(u/\eps^2)dW(u), \qquad M^{s,\eps}_t=\int_0^t \sin(u/\eps^2)dW(u).
\end{align*}
The term in \eqref{add:newlem:supbound:eq:2_split2_bound} can be bounded above by
\begin{align}\label{add:newlem:supbound:eq:2_split2_bound2}
\sigma \sqrt{\Psiz^*\Psiz} \,\left( \,\modcont(\eps^2r,T;M^{c,\eps}) \,+\, \modcont(\eps^2r,T;M^{s,\eps}) \,\right).
\end{align}
To compute $\mbbP[\sup_{s\in [0,T]}\Upsilon^{\eps}_s \,\geq \eps^a]$ we use the identity for nonnegative random variables $X_i$
$$\mbbP[\sum_{j=1}^5X_j \geq \eps^a] \,\,\,\leq\,\,\, \sum_{j=1}^5\mbbP[X_j \geq \frac15\eps^a] \,\,\,\leq\,\,\, 5\eps^{-a}\sum_{j=1}^5\expt[X_j].$$
Now, $\expt[\sup_{t\in [0,T]}|W(t)|]\leq C_{bdg}\sqrt{T}$ by Burkholder-Davis-Gundy inequality (theorem 3.3.28 in \cite{KaratzasShreeve}). Also
$$\expt \int_0^T|\sigma W(u)|du\,\, \leq\,\, T\expt\sup_{t\in[0,T]}|\sigma W(u)|\,\,\leq\,\, \sigma C_{bdg}T^{3/2}.$$
Using lemma 3 of \cite{FischerNappo}, $\expt \,\modcont(r \eps^\delta,T;W) \leq C_{\modcont}\sqrt{{r\eps^\delta}\ln\left(\frac{T}{r\eps^\delta}\right)}$. Using Theorem 1 of \cite{FischerNappo} there exists constants $C_{\modcont}^c$ and $C_{\modcont}^s$ such that $\expt \,\modcont( \eps^2r,T;M^{c,\eps}) \leq C_{\modcont}^c\sqrt{{\eps^2r}\ln\left(\frac{T}{\eps^2r}\right)}$ and $\expt \,\modcont( \eps^2r,T;M^{s,\eps}) \leq C_{\modcont}^s\sqrt{{\eps^2r}\ln\left(\frac{T}{\eps^2r}\right)}$. Hence
\begin{align*}
\mbbP[\sup_{s\in [0,T]}\Upsilon^{\eps}_s \,\geq \eps^a] \,\,\,\leq &\,\,\,C_1\eps^{-a}e^{-\kappa r \eps^{\delta-2}}\,+\,C_2\eps^{-a-2}e^{-\kappa r \eps^{\delta-2}} \\
& \quad + C_3\eps^{-a}\sqrt{{r\eps^\delta}\ln\left(\frac{T}{r\eps^\delta}\right)} \,\,+\,\,C_4\,\eps^{-a}\sqrt{{\eps^2r}\ln\left(\frac{T}{\eps^2r}\right)}
\end{align*}
where $C_{1}=5\sigma K C_{bdg}\sqrt{T}$, $C_2=5\sigma ||L_0||KC_{bdg}T^{3/2}$, $C_3=10\sigma K\left(1+\frac{||L_0||}{2\kappa}\right)C_{\modcont}$, $C_4=5\sigma \sqrt{\Psiz^*\Psiz}(C_{\modcont}^c+C_{\modcont}^s)$. For small enough $\eps$, the term with $C_3$ dominates and we have
\begin{align*}
\mbbP[\sup_{s\in [0,T]}\Upsilon^{\eps}_s \,\geq \eps^a] \,\,\,\leq &\,\,\, 2C_3\eps^{-a}\sqrt{{r\eps^\delta}\ln\left(\frac{T}{r\eps^\delta}\right)}.
\end{align*}
Given $a\in (0,1)$ if we choose $\delta \in (2a,2)$ the RHS of the above equation goes to zero as $\eps \to 0$.
\end{proof}

Fix a $H^*\in \R^+$ and let $$\Ssp:=\{\eta \in P_{\Lambda}\,:\,\ham(\eta)<H^*\}.$$
Assume that the initial condition $\hat{\pj}_0\Xeps$ is such that $\pi \hat{\pj}_0\Xeps \in \Ssp$. Define the stopping time
$$\stopt := \inf \{t\geq 0 \,:\, \pi\hat{\pj}^{\eps}_t\Xeps\not\in \Ssp \}.$$

\begin{prop}\label{add:newprop:supboundGron}
Define $\beta^{\eps}_s:=||(I-\pi)\hat{\pj}^{\eps}_s \Xeps||$. Then there exists constants $C,\hat{C}$ and $\hat{\eps}_{(a,\delta)}$ such that, given any $a\in [0,1)$,  for $\eps < \hat{\eps}_{(a,\delta)}$ 
\begin{align}\label{add:newprop:supboundGron_statement}
\mbbP\left[\sup_{s\in [0,T\wedge \stopt]}\left(\beta^{\eps}_s -\beta^{\eps}_0(1+\frac12{C}^2s^2)e^{-\kappa s/\eps^2}\right)\,\,\geq \,\,  2\eps^a \right] \,\,\,\leq\,\,\,\hat{C}\eps^{-a}\sqrt{{r\eps^\delta}\ln\left(\frac{T}{r\eps^\delta}\right)}, \qquad \quad \delta \in (2a,2). 
\end{align}
\end{prop}
\begin{proof}
Using the variation of constants formula, we have
\begin{align}\label{eq:vocnonlinstable}
||(I-\pi)\hat{\pj}^{\eps}_s \Xeps||\,\,\,\leq &\,\,\,||\Th(s/\eps^2)(I-\pi)\hat{\pj}^{\eps}_0 \Xeps|| \\ \notag
& +\,||\int_0^s\Th(\frac{s-u}{\eps^2})(I-\pih)\Ind G(\hat{\pj}^{\eps}_u \Xeps)du ||  \\ \notag
& +\,||\int_0^s\Th(\frac{s-u}{\eps^2})(I-\pih)\Ind \sigma dW(u)||.
\end{align}
For the first term on the RHS of \eqref{eq:vocnonlinstable}, using part 5 of lemma \ref{lem:collec_ext_res}, we have,
\begin{align}\label{eq:vocnonlinstable_esti1}
||\Th(s/\eps^2)(I-\pi)\hat{\pj}^{\eps}_0 \Xeps|| \,\,\,\leq\,\,\, K||(I-\pi)\hat{\pj}^{\eps}_0 \Xeps||e^{-\kappa s/\eps^2}=K\beta^\eps_0e^{-\kappa s/\eps^2}.
\end{align}
For the second term on the RHS of \eqref{eq:vocnonlinstable}, using Lipshitz condition on $G$ and that $\pi \hat{\pj}^{\eps}_u\Xeps$ is bounded for $u\in [0,T\wedge \stopt]$, we have,
\begin{align}\label{eq:vocnonlinstable_esti2}
||\int_0^s & \Th(\frac{s-u}{\eps^2})(I-\pih)\Ind  G(\hat{\pj}^{\eps}_u \Xeps)du || \\ \notag
& \leq \,\,\,\int_0^s||\Th(\frac{s-u}{\eps^2})(I-\pih)\Ind  ||\,|G(\pi\hat{\pj}^{\eps}_u \Xeps)|\,du \\ \notag
& \qquad \quad + K_G\int_0^s||\Th(\frac{s-u}{\eps^2})(I-\pih)\Ind  ||\,||(I-\pi)(\hat{\pj}^{\eps}_u \Xeps)||\,du \\ \notag
& \leq \,\,\, C\int_0^s e^{-\kappa(s-u)/\eps^2}du \quad + \quad C\int_0^s e^{-\kappa(s-u)/\eps^2}\,||(I-\pi)(\hat{\pj}^{\eps}_u \Xeps)||\,du \\ \notag
& \leq \,\,\,C\eps^2\,\,+\,\,C\int_0^s e^{-\kappa(s-u)/\eps^2}\,\beta^\eps_u\,du.
\end{align}
Hence for $s\in [0,T\wedge \stopt]$
\begin{align*}
\beta^\eps_s-\left(K\beta^\eps_0e^{-\kappa s/\eps^2}\,\,+\,\,C\eps^2\,\,+\,\,C\int_0^s e^{-\kappa(s-u)/\eps^2}\,\beta^\eps_u\,du\right)\,\,\leq\,\,||\int_0^s\Th(\frac{s-u}{\eps^2})(I-\pih)\Ind \sigma dW(u)||.
\end{align*}
For the  RHS of the above inequality we use lemma \ref{add:newlem:supbound}. Then we have the following statement: for any $a\in [0,1)$, there exists constants $\hat{C}$ and $\eps_{(a,\delta)}$ such that for $\eps < \eps_{(a,\delta)}$ 
\begin{align}\label{add:newprop:supbound_befGron}
\mbbP\bigg[\forall s\in [0,T\wedge\stopt],\quad \beta^{\eps}_s \,\,\leq\,\,K\beta^\eps_0e^{-\kappa s/\eps^2}\,\,&+\,\,C\eps^2\,\,+\,\,C\int_0^s e^{-\kappa(s-u)/\eps^2}\,\beta^\eps_u\,du+ \eps^a \bigg] \\ \notag
& \geq\,\,\,1-\hat{C}\eps^{-a}\sqrt{{r\eps^\delta}\ln\left(\frac{T}{r\eps^\delta}\right)}, \qquad \quad \delta \in (2a,2). 
\end{align}
Using Gronwall kind of inequality (see Theorem 1.5 on page 7 of \cite{Bainov}) we have that LHS of \eqref{add:newprop:supbound_befGron} is bounded above by
\begin{align}\label{add:newprop:supbound_aftGron}
\mbbP\bigg[\forall s\in [0,T\wedge\stopt],\quad \beta^{\eps}_s \,\,\leq\,\,\beta^{\eps}_0(1+\frac12{C}^2s^2)e^{-\kappa s/\eps^2} + (1+\eps^2{C}^2/\kappa^2)({C}\eps^2+\eps^a) \bigg].
\end{align}
Let $\eps_*$ be such that $\forall \eps<\eps_*$, $(1+\eps^2{C}^2/\kappa^2)({C}\eps^2+\eps^a)<2\eps^a$. Choose $\hat{\eps}_{(a,\delta)}=\eps_{(a,\delta)}\wedge \eps_*$.
\end{proof}



\begin{lem}\label{add:newlem:bound}
Let $\stabsol$ be as in definition \ref{def:stabsol}. Then
\begin{align}\label{add:newlem:lhs}
\left|\left|\int_0^s\Th(\frac{s-u}{\eps^2})(I-\pih)\Ind \sigma dW(u)\right|\right|\,\,\leq\,\,& \\
\label{add:newlem:term1}
 &\left|\int_0^{(s-\eps^2r)\vee 0}\stabsol\left(\frac{s-\eps^2r-u}{\eps^2}\right)\sigma dW(u)\right| \\ \label{add:newlem:term2}
& + \sup_{v\in[(s-\eps^2r)\vee 0,s]}\left|\int_{(s-\eps^2 r)\vee 0}^v \stabsol(0)\sigma dW(u)\right| \\ \label{add:newlem:term3}
& + \sup_{v\in[(s-\eps^2r)\vee 0,s]}\left|\int_{(s-\eps^2 r)\vee 0}^v d\tau \int_0^\tau \frac{1}{\eps^2}\stabsol'\left(\frac{\tau -u}{\eps^2}\right)\sigma dW(u)\right| \\ \label{add:newlem:term4}
& + \sup_{v\in[(s-\eps^2r)\vee 0,s]}\left|\sqrt{\Psiz^*\Psiz}\int_{v}^s \cos(u/\eps^2)\sigma dW(u)\right| \\ \label{add:newlem:term5}
& + \sup_{v\in[(s-\eps^2r)\vee 0,s]}\left|\sqrt{\Psiz^*\Psiz}\int_{v}^s \sin(u/\eps^2)\sigma dW(u)\right|. 
\end{align}
Also,
\begin{align}\label{add:newlem:avgbound}
\expt\left|\left|\int_0^s\Th(\frac{s-u}{\eps^2})(I-\pih)\Ind \sigma dW(u)\right|\right|\,\,\leq\,\,\eps \sigma \left(C_{bdg}\sqrt{r}(2\sqrt{\Psiz^*\Psiz}+|h(0)|)+\frac{K}{\sqrt{2\kappa}}(1+r||L_0||)\right).
\end{align}
\end{lem}
\begin{proof}
Assume $s>\eps^2 r$. LHS of \eqref{add:newlem:lhs} is bounded above by
\begin{align}\label{add:newlem:prfpart:term1}
\sup_{\theta \in [-r,0]}&\left|\int_0^{(s+\eps^2\theta)}\stabsol\left(\frac{s+\eps^2\theta -u}{\eps^2}\right)\sigma dW(u)\right| \\ \label{add:newlem:prfpart:term2}
&+ \sup_{\theta \in [-r,0]}\left|\int_{(s+\eps^2\theta)}^s\pih \Ind\left(\frac{s+\eps^2\theta -u}{\eps^2}\right)\sigma dW(u)\right|.
\end{align}

The term in \eqref{add:newlem:prfpart:term1} can be rewritten as
\begin{align}\label{add:newlem:prfpart:term1_rewrite}
\sup_{v \in [s-\eps^2r,s]}&\left|\int_0^{v}\stabsol\left(\frac{v -u}{\eps^2}\right)\sigma dW(u)\right|.
\end{align}
For $v\in [s-\eps^2r,s]$, we have
\begin{align}\label{add:newlem:prfpart:useful1}
\stabsol\left(\frac{v -u}{\eps^2}\right)=\stabsol\left(\frac{s-\eps^2r-u}{\eps^2}\right)\mathbf{1}_{\{u<s-\eps^2r\}}+\stabsol(0)\mathbf{1}_{\{u\geq s-\eps^2r\}}+\int_{(s-\eps^2r)\vee u}^v\frac{1}{\eps^2}\stabsol'\left(\frac{\tau -u}{\eps^2}\right)d\tau.
\end{align}
Using the above in \eqref{add:newlem:prfpart:term1_rewrite} and then changing the order of integration for the term involving $\stabsol'$ we have that \eqref{add:newlem:prfpart:term1} is bounded above by RHS of \eqref{add:newlem:term1} $+$ \eqref{add:newlem:term2} $+$ \eqref{add:newlem:term3}.

For the term in \eqref{add:newlem:prfpart:term2} we make use of the following facts: 
\begin{align*}
\pih\Ind(v-u)&=\Psiz_1 \cos(v-u)+\Psiz_2\sin(v-u) \\
&=(\Psiz_1\cos v+\Psiz_2\sin v)\cos u + (\Psiz_1\sin v-\Psiz_2\cos v)\sin u, 
\end{align*}
and $|\Psiz_1\cos v+\Psiz_2\sin v|\leq \sqrt{\Psiz_1^2+\Psiz_2^2}=\sqrt{\Psiz^*\Psiz}$. Using these it is easy to see that the term in \eqref{add:newlem:prfpart:term2} is bounded above by  \eqref{add:newlem:term4} $+$ \eqref{add:newlem:term5}.

Same method can be employed for $s\leq \eps^2r$.

Now we prove \eqref{add:newlem:avgbound}. Using exponential decay of $|\stabsol|$ we have
\begin{align*}
\expt \eqref{add:newlem:term1} \,\,\leq\,\,\sqrt{\expt \eqref{add:newlem:term1}^2} \,\,&\leq\,\,\sqrt{\int_0^{(s-\eps^2r)\vee 0}\left(\stabsol\left(\frac{s-\eps^2r-u}{\eps^2}\right)\right)^2\sigma^2 du} \\
&\leq\,\,\sqrt{\int_0^{(s-\eps^2r)\vee 0}K^2e^{-2\kappa(s-\eps^2r-u)/\eps^2}\sigma^2 du} \,\,\leq\,\,\eps\sigma \frac{K}{\sqrt{2\kappa}}. 
\end{align*}
 Using exponential decay of $|\stabsol'|$ we have
\begin{align*}
\expt \eqref{add:newlem:term3} \,\,&\leq \,\,\int_{(s-\eps^2r)\vee 0}^s d\tau \,\expt\left|\int_0^\tau \frac{1}{\eps^2}\stabsol'\left(\frac{\tau-u}{\eps^2}\right)\sigma dW(u)\right| \\
&\leq \,\,\int_{(s-\eps^2r)\vee 0}^s d\tau \,\sqrt{\int_0^{\tau}\left(\frac{1}{\eps^2}\stabsol'\left(\frac{\tau-u}{\eps^2}\right)\right)^2\sigma^2 du} \\
&\leq \,\,\int_{(s-\eps^2r)\vee 0}^s d\tau \,\sqrt{\int_0^{\tau}\frac{1}{\eps^4}K^2||L_0||^2e^{-2\kappa(\tau-u)/\eps^2}\sigma^2 du} \,\,\leq\,\, \eps\sigma \frac{K}{\sqrt{2\kappa}}r||L_0||.
\end{align*}
Using Burkholder-Davis-Gundy inequality $\exists\, C_{bdg}$ such that $\expt \eqref{add:newlem:term2} \leq C_{bdg}\sigma |h(0)|\sqrt{\eps^2r}$, $\expt \eqref{add:newlem:term4} \leq C_{bdg}\sigma \sqrt{\Psiz^*\Psiz}\sqrt{\eps^2r}$ and $\expt \eqref{add:newlem:term5} \leq C_{bdg}\sigma \sqrt{\Psiz^*\Psiz}\sqrt{\eps^2r}$. Combining the above results we have \eqref{add:newlem:avgbound}.
\end{proof}

\begin{prop}\label{prop:stablenormtozero}
For any $\nu<1$,
\begin{align}\label{eq:lem:stablenormtozero}
\lim_{\eps \to 0}\eps^{-\nu}\expteps  \int_0^{t\wedge \stopt} ||(I-\pi)\hat{\pj}^{\eps}_s \Xeps||\,ds =0.
\end{align}
\end{prop}
\begin{proof}
Let $\beta^{\eps}_s:=||(I-\pi)\hat{\pj}^{\eps}_s \Xeps||$. 
Using the variation of constants formula \eqref{eq:vocnonlinstable}, and the estimates \eqref{eq:vocnonlinstable_esti1}, \eqref{eq:vocnonlinstable_esti2}, \eqref{add:newlem:avgbound} we have
\begin{align}\label{prop:stablenormtozero_esti1}
\expteps\int_0^{t\wedge \stopt}\beta^{\eps}_sds \,\,\,\,\leq &\,\,\,\,K\beta^{\eps}_0 \int_0^{t\wedge \stopt} e^{-\kappa s/\eps^2}ds \\ \label{prop:stablenormtozero_esti2} &\qquad +\,\int_0^{t\wedge \stopt}C\eps^2\,ds \,+\, C\expteps\int_0^{t\wedge \stopt}\int_0^s e^{-\kappa(s-u)/\eps^2}\,\beta^{\eps}_u\,du\,ds  \\ \label{prop:stablenormtozero_esti3}
& \qquad + \,\int_0^{t\wedge \stopt}C\eps\,ds.
\end{align}
Evaluating the above integrals and changing the order of integration of second term of \eqref{prop:stablenormtozero_esti2} we have
\begin{align}\notag
\expteps\int_0^{t\wedge \stopt}\beta^{\eps}_sds \,\,\,\,\leq &\,\,\,\,\eps^2\frac{K}{\kappa}\beta^{\eps}_0 + \eps^2Ct +  C \, \int_0^{t\wedge \stopt} du\,\expteps \beta^{\eps}_u \int_u^{t\wedge \stopt}e^{-\kappa(s-u)/\eps^2}ds + \eps Ct \\ \label{eq:propstableesti_use_for_quad}
\leq &\,\,\,\,\eps^2C(1+t) + \eps Ct + C \, \eps^2\int_0^{t\wedge \stopt} du\,\expteps \beta^{\eps}_u 
\end{align}
Hence, for sufficiently small $\eps$,
\begin{align}\label{eq:stablespaceestimate}
\expteps  \int_0^{t\wedge \stopt} ||(I-\pi)\hat{\pj}^{\eps}_s \Xeps||\,ds\,\,\, \leq\,\,\,Ct\eps +C\eps^2.
\end{align} 
And hence, for any $\nu<1$, the statement \ref{eq:lem:stablenormtozero} holds.
\end{proof}

\section{Main result}\label{sec:mainresult}
Recall that our aim is to study the weak convergence, as $\eps \to 0$, of the law of $\hprc^{\eps}(t):=\ham(\hat{\pj}^{\eps}_t \Xeps)$ where $\ham(\eta):=\frac12\la \eta,\Psi \ra^*\la \eta,\Psi \ra$. We use the martingale problem technique and hence, as a starting point, we need a result similar to theorem \ref{thm:qtfmart_general}.

We alter the definition \ref{def:A0_def} to suit the system \eqref{eq:main_in_intgl_form_timesclaechange_drophat}. Let $C_b$ is the Banach space of all bounded continuous functions.
Define an operator $\genonqtf$ on $C_b$ with $\dom(\genonqtf)=\tau_q$ (quasi-tame functions) as follows: Let $\qtf \in \dom(\genonqtf)$ be of the form \eqref{eq:qtf_form}. Then 
\begin{align}\label{eq:A0_def_timescalechange}
(\genonqtf \qtf)(\eta)&:= \frac{1}{\eps^2}(\tgenonqtf \qtf)(\eta)  \,\,+\,\, (\sgenonqtf \qtf)(\eta), \\ 
(\tgenonqtf \qtf)(\eta)&:=\sum_{j=0}^{k-1}\bigg( f_j(\eta(0))g_j(0)-f_j(\eta(-r))g_j(-r)-[f_j,g_j'|\eta] \bigg)\partial_j\qtfr \,\,+\,\, L_0(\eta) \partial_k\qtfr \notag \\ \label{eq:transgendef}
&=\,\,\frac{d}{dt}\bigg|_{t=0}\qtf(T(t)\eta), \\
(\sgenonqtf \qtf)(\eta)&:=G(\eta)\partial_k \qtfr\,+\,\frac12 F^2(\eta)\partial_k^2 \qtfr.
\end{align}
In \eqref{eq:transgendef}, $T(t)$ is the semigroup from section \ref{sec:unpertprob}, and the equality follows from lemma \ref{lem:appliItoformula}.

The function $\ham(\eta)$, which was defined by $\ham(\eta)=\frac12\la \eta,\Psi \ra^*\la \eta,\Psi \ra$, is not a quasi-tame function. Recall the bilinear form \eqref{eq:bilinform}. As a corollary of Riesz representation theorem, there exists a bounded variation (BV) function $\mu:[-r,0]\to \R$ (unique when normalized, see Theorem 1.1.1 in \cite{Diekmanbook}) such that $L_0\eta=\int_{-r}^0 d\mu(\theta)\eta(\theta)$. Hence \eqref{eq:bilinform} can be written as
\begin{align*}
\la \eta,\psi \ra &= \eta(0)\psi(0)-\int_{-r}^0 d\mu(\theta)\int_0^{\theta}\eta(u)\psi(u-\theta)du \\
&=\eta(0)\psi(0)+\int_{-r}^0 \eta(u)\left(\int_{-r}^{u}\psi(u-\theta)d\mu(\theta)\right) du.
\end{align*}
Define $f_1(x)=f_2(x)=x$, $g_i(u)=\int_{-r}^u \Psi_i(u-\theta)d\mu(\theta)$ and $\qtfr(x,y,z)=\frac12 (\Psi_1(0)z+x)(\Psi_2(0)z+y)$. Then $\ham(\eta)=\qtfr([f_1,g_1|\eta],[f_2,g_2|\eta],\eta(0))$. Note that $f_i$ are not bounded and hence $\ham$ is not a quasi-tame function. 

The BV function $\mu$ can have jump discontinuities\footnote{For a BV function \emph{only} jump discontinuities are possible. The set of discontinuous points is at most countable. See theorem 1.2 and appendix 1 of \cite{Diekmanbook}.} and so the functions $g_i$ can fail to be continuous. For example, consider the unperturbed system 
$$\dot{x}\,=\,\alpha x(t-1)\,+\,\beta\int_{-1/2}^0x(t+\theta)d\theta \,+\,\gamma x(t-1/4).$$
For this system $d\mu(\theta)=\alpha \delta_{-1}(\theta)+\beta \mathbf{1}_{[-1/2,0]}(\theta)d\theta + \gamma \delta_{-1/4}(\theta)$ where $\delta_{a}(\cdot)$ is the Dirac-delta at $\theta=a$. Assume $\alpha,\beta,\gamma$ are such that the assumption \ref{ass:assumptondetsys} holds. Evaluating $g_i(u)=\int_{-1}^u \Psi_i(u-\theta)d\mu(\theta)$ we have
$$g_i(u)=\alpha \Psi_i(u+1)\,+\, \mathbf{1}_{\{u\in [-1/2,0]\}}\beta\int_{-1/2}^u\Psi_i(u-\theta)d\theta \,+\,  \mathbf{1}_{\{u\in [-1/4,0]\}}\gamma\Psi_i(u+1/4).$$
From the above it can be easily seen that $g_i$ are piecewise $C^1$ if $\gamma=0$. If $\gamma\neq 0$ they fail to be continuous. However this situation can be alleviated by making alterations to lemma \ref{lem:howqtfevolves} as follows: define $\tilde{g}_i(u)=\mathbf{1}_{\{u\in [-1/4,0]\}} \gamma\Psi_i(u+1/4)$ and then
\begin{align*}
[f,\tilde{g}_i|\pj_tX]&-[f,\tilde{g}_i|\pj_0 X] =\int_0^t\left\{ f(\pj_uX(0))\tilde{g}_i(0)-f(\pj_uX(-1/4))\tilde{g}_i(-1/4)-[f,\tilde{g}'|\pj_uX]\right\}du.
\end{align*}
Hence the jump discontinuities in $\mu$ do not create any problems when interpreted properly.

The lemmas \ref{lem:howqtfevolves} and \ref{lem:appliItoformula} are true even in the case that $\qtf:\C \to \R$ is of the form \eqref{eq:qtf_form} with $f_j \in C(\R;\R)$ and $\qtfr \in C^2(\R^k;\R)$ and $g_j$ continuous and piecewise $C^1$.   Denote the set of such functions $\qtf$ by $\tauqext$. Define an operator $\genonqtfext$ with domain as $\tauqext$ and action same as \eqref{eq:A0_def_timescalechange}.

Now it is clear that $\ham \in \dom({\genonqtfext})$. 

We write $\Psiz$ for $\Psi(0)$, and $\Psiz^*$ is the transpose of $\Psiz$.

\begin{prop}\label{prop:Hmart}
\begin{align*}
(\genonqtfext \, \ham)(\eta)\,=\,(\sgenonqtf  \ham)(\eta)\,=\,G(\eta)\Psiz^*\la \eta,\Psi\ra\,+\,\frac12 F^2(\eta)\Psiz^*\Psiz. 
\end{align*}
The process $M^{\eps}_t$ defined by
\begin{align}
M_t^{\eps}:= \ham(\hat{\pj}^{\eps}_{t\wedge \stopt}\Xeps)- \ham(\hat{\pj}^{\eps}_0\Xeps)-\int_0^{t\wedge \stopt}(\sgenonqtf  \ham)(\hat{\pj}^{\eps}_u \Xeps)du
\end{align}
is a $\F_t$ martingale with quadratic variation given by $\int_0^{t\wedge \stopt} Q(\hat{\pj}^{\eps}_s \Xeps)ds$ where $Q(\eta)=\left(F(\eta) \Psiz^*\la \eta,\Psi\ra\right)^2$.
\end{prop}
\begin{proof}
The first result follows from $(\tgenonqtf \ham)(\eta)=\frac{d}{dt}\big|_{t=0}\ham(T(t)\eta)=0$ (see remark \ref{rmk:sayingHevolvesslow}). Application of lemmas \ref{lem:howqtfevolves} and \ref{lem:appliItoformula} shows that $$M_t^{\eps}=\int_0^{t\wedge \stopt}F(\hat{\pj}^{\eps}_u \Xeps) \Psiz^*\la \hat{\pj}^{\eps}_u \Xeps,\Psi\ra dW(u) = \sigma \int_0^{t} \Psiz^*\la \hat{\pj}^{\eps}_u \Xeps,\Psi\ra \,\mathbf{1}_{\{u \leq \stopt \}}\,dW(u).$$
Note that $$\expt \int_0^t \big(\Psiz^*\la \hat{\pj}^{\eps}_u \Xeps,\Psi\ra \,\mathbf{1}_{\{u \leq \stopt \}} \big)^2 du \,\, \leq \,\, \expt \int_0^t 2H^*(\Psiz^*\Psiz)\,du \,\,\leq\,\, 2H^*(\Psiz^*\Psiz)t.$$
Hence (see definition 3.2.9 and proposition 3.2.10 of \cite{KaratzasShreeve}) $M_t^{\eps}$ is a $\F_t$ martingale with quadratic variation $\int_0^{t\wedge \stopt} \left(\sigma \Psiz^*\la \hat{\pj}^{\eps}_u \Xeps,\Psi\ra\right)^2 ds$.
\end{proof}

\begin{prop}\label{prop:hofHmart}
For $\htf$ any $C^2(\R)$ function, $\htf \circ \ham \in \dom(\genonqtfext)$.
\begin{align*}
(\genonqtfext (\htf\circ \ham))(\eta)&=(\sgenonqtf (\htf\circ \ham))(\eta)\\
&=G(\eta)\htf'\big|_{\ham(\eta)}\Psiz^*\la \eta,\Psi\ra\,+\,\frac12 F^2(\eta)\bigg( \htf'\big|_{\ham(\eta)}\Psiz^*\Psiz \,\,+\,\,\htf''\big|_{\ham(\eta)}\left( \Psiz^*\la \eta,\Psi\ra\right)^2 \bigg).
\end{align*}
The process $M^{\eps}_t$ defined by
\begin{align}
M_t^{\htf,\eps}= \htf \circ \ham(\hat{\pj}^{\eps}_{t\wedge \stopt}\Xeps)- \htf \circ \ham(\hat{\pj}^{\eps}_0\Xeps)-\int_0^{t\wedge \stopt}(\sgenonqtf  (\htf \circ \ham))(\hat{\pj}^{\eps}_u \Xeps)du
\end{align}
is a $\F_t$ martingale with quadratic variation given by $\int_0^{t\wedge \stopt} Q(\hat{\pj}^{\eps}_s \Xeps)ds$ where $$Q(\eta)=\left(F(\eta) \htf'\big|_{\ham(\eta)}\Psiz^*\la \eta,\Psi\ra\right)^2.$$
\end{prop}
\begin{proof}
Similar to proof of proposition \ref{prop:Hmart}. Note that $h'$ and $h''$ are bounded on $[0,H^*]$.
\end{proof}

We use the semigroup $T(t)$ to generate an equivalence relation on $P_{\Lambda}$, i.e. $$\eta_1 \sim \eta_2, \text{ if }  \exists t\in \R \text{ s.t. } T(t)\eta_1=\eta_2.$$ 
For $\eta \in P_{\Lambda}$, let $$[\eta]=\{\zeta \in P_{\Lambda}\,:\, \zeta \sim \eta\}$$ be the equivalence class of $\eta$ and define $\eqvpj(\eta):=[\eta]$.
Define $\eqvtoH: (\bar{\Ssp}/\sim) \to \R$ by $\eqvtoH([\eta]):=\ham(\eta).$ The image of $\eqvtoH$ is the closed interval $[0,H^*]$.

\vspace{24pt}
Define the averaging operator $\Aavg:C(\bar{\Ssp})\to C(\bar{\Ssp}/\sim)$ 
$$(\Aavg \varphi)([\eta]):=\frac{1}{\tmpr}\int_0^{\tmpr}\varphi(T(s)\eta)ds.$$
where $\tmpr=\tmprval$.

Define $b_{H}:[0,H^*]\to \R$ and $\sigma^2_{H}:[0,H^*]\to \R$ by
$$b_H \circ \eqvtoH = \Aavg \bigg(G(\cdot)E(\cdot)+\frac12 F^2(\cdot)\Psiz^*\Psiz \bigg), \qquad \quad \sigma^2_H \circ \eqvtoH = \Aavg \bigg( F^2(\cdot) E^2(\cdot)  \bigg),$$
where $E(\eta)=\Psiz^*\la \eta,\Psi\ra$.

Now we compute the averaged drift and diffusion coefficients $b_H$ and $\sigma^2_H$. A representative element $\eta \in P_{\Lambda}$ from the equivalence class $[\eta]$ whose $\ham(\eta)$ equals $\hlev$ can be taken as $\eta(\theta)=\sqrt{2\hlev}\cos\omega_c\theta$ for $\theta \in [-r,0]$. Then we have
$$b_H(\hlev)=b_H^{(1)}(\hlev)+b_H^{(2)}(\hlev)$$
where
\begin{align}\label{eq:avgdr1}
b_H^{(1)}(\hlev)&=\frac{1}{\tmpr}\int_0^{\tmpr}\frac12F^2(T(s)\sqrt{2\hlev}\cos\omega_c\cdot)\Psiz^*\Psiz\,ds\,\,=\,\,\frac12\sigma^2\Psiz^*\Psiz \,\,=:\,\, \consa_b,
\end{align}
\begin{align}\label{eq:avgdr2}
b_H^{(2)}(\hlev)&=\frac{1}{\tmpr}\int_0^{\tmpr}G(T(s)\sqrt{2\hlev}\cos\omega_c\cdot)\,\Psiz^*\la T(s)\sqrt{2\hlev}\cos\omega_c\cdot,\Psi\ra \,ds.
\end{align}
\begin{align}\label{eq:avgdiff}
\sigma^2_{H}(\hlev)&=\frac{1}{\tmpr}\int_0^{\tmpr}F^2(\sqrt{2\hlev}T(s)\cos\omega_c\cdot) \left( \Psiz^*\la \sqrt{2\hlev}T(s)\cos\omega_c\cdot,\Psi\ra\right)^2 \,ds \\ \notag
&=\sigma^2 2\hlev\frac{1}{\tmpr}\int_0^{\tmpr}\left( \Psiz^*\la T(s)\cos\omega_c\cdot,\Psi\ra\right)^2 \,ds \,\, =:\,\, \hlev \consa_d.
\end{align}
Note that $\consa_b, \consa_d$ are both positive. Further, $2\consa_b=\consa_d$ because, using \eqref{eq:TPhi_eq_PhieB}
\begin{align*}
\Psiz^*\la T(t)\cos\omega_c\cdot,\Psi\ra \,\,\,&=\,\,\,\Psiz^*\la \cos\omega_ct\,\Phi_1-\sin\omega_ct\,\Phi_2,\Psi\ra \\
&=\cos\omega_ct\Psiz^*\la \Phi_1,\Psi\ra -\sin\omega_ct\Psiz^*\la \Phi_2,\Psi\ra \\
&=\cos\omega_ct\Psiz_1 -\sin\omega_ct\Psiz_2,
\end{align*}
$$\consa_d=2\sigma^2 \frac{1}{\tmpr}\int_0^{\tmpr} (\cos\omega_ct\Psiz_1 -\sin\omega_ct\Psiz_2)^2\,ds = \sigma^2( \Psiz_1^2+\Psiz_2^2 )= \sigma^2 \Psiz^*\Psiz=2\consa_b.$$

\begin{rmk}\label{rmk:crucial}
We will use later the fact that $\sup_{\hlev \in [0,H^*]}\frac{|b_H^{(2)}(\hlev)|}{\hlev}<\infty$. This can be proved, using Lipshitz condition on $G$ and \eqref{eq:TPhi_eq_PhieB} as follows:
\begin{align*}
|b_H^{(2)}(\hlev)| & \leq \bigg|\frac{1}{\tmpr}\int_0^{\tmpr}\big\{ G(0) + \big( G(T(s)\sqrt{2\hlev}\cos\omega_c\cdot)-G(0) \big) \big\}\,\Psiz^*\la T(s)\sqrt{2\hlev}\cos\omega_c\cdot,\Psi\ra \,ds\bigg| \\
& \leq \bigg|\frac{1}{\tmpr}\int_0^{\tmpr}G(0) \,\Psiz^*\la T(s)\sqrt{2\hlev}\cos\omega_c\cdot,\Psi\ra \,ds\bigg| \\ & \qquad + \bigg|\frac{1}{\tmpr}\int_0^{\tmpr}\big( G(T(s)\sqrt{2\hlev}\cos\omega_c\cdot)-G(0) \big) \,\Psiz^*\la T(s)\sqrt{2\hlev}\cos\omega_c\cdot,\Psi\ra \,ds\bigg| \\
& \leq 0+\frac{1}{\tmpr}\int_0^{\tmpr}K_G\,||T(s)\sqrt{2\hlev}\cos\omega_c\cdot ||\,\,|\Psiz^*\la T(s)\sqrt{2\hlev}\cos\omega_c\cdot,\Psi\ra| \,ds \\ 
&\leq C\hlev.
\end{align*}
Using similar means, it can also be proved that $b_H^{(2)}$ is Lipshitz. \hfill $\triangle$
\end{rmk}


Define an operator $\igenH$ by\footnote{$\igenH f_H\in C([0,H^*])$ means that $\igenH f_H\in C((0,H^*))$ and the limits $\lim_{\hlev \downarrow 0}\igenH f_H$, $\lim_{\hlev \uparrow H^*}\igenH f_H$ exists and are finite. See chapter 8 section 1 of \cite{Ethier_Kurtz}.}
\begin{align}
\dom(\igenH)=\bigg\{ f_H \in C([0,H^*]) & \cap C^2((0,H^*))\,:\,\notag \\
& \igenH f_H\in C([0,H^*]) \text{ and } \lim_{\hlev \uparrow H^*}(\igenH f_H)(\hlev)=0 \bigg\}, \notag \\ 
\text{for } \hlev \in (0,H^*), \qquad (\igenH f_H)(\hlev)&=b_H(\hlev)\dot{f}_H(\hlev)+\frac12\sigma^2_{H}(\hlev)\ddot{f}_{H}(\hlev). \label{eq:igenHdef_limitproc}
\end{align}

The main result of this paper is the following:
\begin{thm}\label{prop:mainresultaddnoise}
Let the process $\Xeps$ be given by \eqref{eq:main_in_intgl_form_timesclaechange_drophat}. Define 
$$\hprc^{\eps}(t):=\ham(\hat{\pj}^{\eps}_t \Xeps) \text{ where } \ham(\eta):=\frac12\la \eta,\Psi \ra^*\la \eta,\Psi \ra.$$
The law of the process $\{ \hprc^\eps({t\wedge\stopt})\,;\,t\geq 0\}$
converges weakly as $\eps \to 0$ to the law of $\{\avgHproc(t\wedge \stoptH) \,;\,t\geq 0\}$ where $\avgHproc$ is the solution of the SDE
$$d\avgHproc(t)=b_H(\avgHproc(t))dt+\sigma_H(\avgHproc(t))dW(t), \qquad \quad \avgHproc(0)=\ham(\xi),$$
and where $\stoptH:=\inf\{t\geq 0\,:\,\avgHproc(t)\geq H^*\}$.

The process $\{\avgHproc(t\wedge \stoptH) \,;\,t\geq 0\}$ is a Markov process whose generator (of the transition semigroup) is an extension of $\igenH$ given by \eqref{eq:igenHdef_limitproc}. 
\end{thm}

\begin{rmk}\label{rmk:mainproofmotiv}
The proof consists of three steps:
\begin{itemize}
\item show that the laws of $\hprc^\eps$ are tight. Then, by Prohorov's theorem, there exists at least one cluster point for the sequence of laws of $\hprc^\eps$, in the weak topology
of probability measures on $C([0,\infty);\R)$.
 
\item for any $f_H\in \dom(\igenH)$, and any $\F_s$ measurable bounded functional $\Theta_s$ of $\hprc^\eps$ show that
\begin{align}
\lim_{\eps \to 0}\expt \bigg[\bigg(f_H(\ham(\hat{\pj}^{\eps}_{t\wedge\stopt} \Xeps))&-f_H(\ham(\hat{\pj}^{\eps}_{s\wedge\stopt} \Xeps))  \notag \\
&  -\int_{s\wedge\stopt}^{t\wedge\stopt}(\igenH f_H)(\ham(\hat{\pj}^{\eps}_u \Xeps))\,du \bigg)\Theta_s(\hprc^\eps)\bigg]=0. \label{eq:clustpointsolvesmartprob}
\end{align}
This shows that any cluster point of the sequence of laws of $\hprc^\eps$ solves the martingale problem for $\igenH$.
\item show the uniqueness of martingale problem for $\igenH$.
\end{itemize}

In order to show the second step, we can make use of proposition \ref{prop:hofHmart} for $f_H \in \dom(\igenH)$ and try to average the term $(\sgenonqtf  (f_H \circ \ham))$. Unfortunately, averaging requires two more derivatives than what is available for $f_H$. To address this, in section \ref{sec:behavtestfunc} we obtain a family of smooth functions $f^\eps$ for every given function $f_H$. In section \ref{sec:avging_approxing} we address the issue of approximating $(\igenH f_H)(\ham(\hat{\pj}^{\eps}_u\Xeps))$ with $(\sgenonqtf (f^\eps \circ \ham))(\hat{\pj}^{\eps}_u\Xeps)$.  

The proof of theorem \ref{prop:mainresultaddnoise} is carried out in section \ref{sec:mainresproof}. \hfill $\triangle$ 
\end{rmk}


\section{Approximations of test functions}\label{sec:behavtestfunc}

In this section our aim is to obtain a family of smooth functions $f^\eps$ for any given $f_H\in \dom(\igenH)$. Our approach is this: given $f_H$, mollify $\igenH f_H$ and obtain the solution of $\igenH u=\poiF^\eps$ where $\poiF^\eps$ is the mollified version of $\igenH f_H$. The solution $u$ will serve our purpose. This is made precise in the following lemmas.

\begin{lem}\label{lem:behavetestfunc} 
Assume $\sigma>0$. If $\poiF\in C([0,H^*])$, then there exists a bounded solution of 
$$\igenH u=\poiF, \qquad \hlev \in (0,H^*)$$ 
such that $u\in C^1([0,H^*])$. This solution is unique upto the choice of $u(0)$. Further, there exists a constant $C$ independent of $\poiF$ such that
$$||u||_{C([0,H^*])}\leq |u(0)|+C||\poiF||_{C([0,H^*])}.$$
\end{lem}
\begin{proof}
The function
$$I(\hlev)=\hlev \,\Xi_I(\hlev), \qquad \quad \Xi_I(\hlev)=\exp\left(2\int_1^\hlev \frac{b_H^{(2)}(\xi)}{\consa_d \,\xi }d\xi\right)$$
serves as an integrating factor as it satisfies $\dot{I}=I\frac{b_H}{\frac12\sigma_H^2}$.
Note that (see remark \ref{rmk:crucial}) $\Xi_I \in C^{\infty}(\R_+\cup \{0\})$ and $\Xi_I(1)=1$ and $0<\inf_{[0,H^*]}\Xi_I(\hlev)<\infty$.

The equation can be solved to give
\begin{align}\label{eq:usefulinuniqueness}
\dot{u}(\hlev)=\frac{2}{I(\hlev)}\int_0^\hlev \frac{\poiF(s)}{\sigma_H^2(s)}I(s)ds.
\end{align}
 Integrating, we find that
\begin{align*}
u(\hlev)&=\hat{C}+\int_0^h \frac{2}{s\Xi_I(s)}\int_0^s \frac{\poiF(r)}{\consa_d}\,\Xi_I(r)\,dr\,ds  \\
&=\hat{C}+\hlev \int_0^1 \frac{2}{s\Xi_I(\hlev s)}\int_0^s \frac{\poiF(\hlev r)}{\consa_d}\,\Xi_I(\hlev r)\,dr\,ds\,\,=:\,\,\hat{C}+\hlev \,\Xi_{\poiF}(\hlev). 
\end{align*}
With the above choice of $u$, $\lim_{\hlev \downarrow 0}u(\hlev)=\hat{C}$. The estimate on $||u||_{C([0,H^*])}$ is straight forward.
\end{proof}

\begin{lem}
If $f_H\in \dom_H$, then $f_H \in C^1([0,H^*])$. 
\end{lem}
\begin{proof}
$f_H\in \dom_H$ implies that $\poiF:=\igenH f_H \in C([0,H^*])$. Use Lemma \ref{lem:behavetestfunc} 
\end{proof}

Fix an exponent $\nu<1$. 
\begin{lem}\label{lem:fHapprox}
Let $f_H\in \dom_H$. There is a sequence $\{f^\eps; \eps>0\}$ of elements of $C([0,H^*])\cap C^4([0,H^*])$ such that
\begin{enumerate}
\item $\lim_{\eps\to 0} ||f^\eps-f_H||_{C([0,H^*])}=0$.
\item $\lim_{\eps\to 0} ||\igenH f^\eps-\igenH f_H||_{C([0,H^*])}=0$.
\item $f^\eps \circ \ham \in C^4(\bar{\Ssp})$  and \footnote{Here we are restricting our attention to $\Ssp$ and not whole of $\C$.} $\sup_{\eps>0}\eps^\nu ||f^\eps \circ \ham||_{C^4(\bar{\Ssp})}<\infty$.
\end{enumerate}
\end{lem}
\begin{proof}
Since $\igenH f_H \in C([0,H^*])$, we can find an extension $\bar{\poiF} \in C_c(\R)$ of $\igenH f_H$. Fix a mollifier function $\molli \in C_c^\infty(\R)$ such that $\molli$ is even and $\int_{z\in \R}\molli(z)dz=1$.  Define
$${\bar{\poiF}}^\eps(\hlev):=\eps^{-\nu/4}\int_{\R}\molli\left(\frac{\hlev-z}{\eps^{\nu/4}}\right)\bar{\poiF}(z)dz$$
for all $\hlev \in \R$.  Let $f^\eps$ be the unique bounded solution of $\igenH f^\eps =\bar{\poiF}^\eps$ on $(0,H^*)$ with $f^\eps(0)=f_H(0)$. Then $u^\eps=f^\eps-f_H$ is the unique bounded solution of $\igenH u^\eps =\bar{\poiF}^\eps-\bar{\poiF}$ on $(0,H^*)$ with $u^\eps(0)=0$. By the estimate of lemma \ref{lem:behavetestfunc}, we have that
$$||f^\eps-f_H||_{C([0,H^*])}\leq C||\bar{\poiF}^\eps-{\bar{\poiF}}||_{C([0,H^*])}.$$ 
Because $\lim_{\eps \to 0}||\bar{\poiF}^\eps-{\bar{\poiF}}||_{C([0,H^*])}=0$, we have that $\lim_{\eps\to 0} ||f^\eps-f_H||_{C([0,H^*])}=0$. Further, from the explicit expression for the solution in lemma \ref{lem:behavetestfunc}, and the fact that $\bar{\poiF}^\eps \in C^\infty$, we have that $f^\eps \in C^4([0,H^*])$. This proves part 1. 

Part 2 is just restating the fact that $\lim_{\eps \to 0}||\bar{\poiF}^\eps-{\bar{\poiF}}||_{C([0,H^*])}=0$.

Part 3: 
Since we have $\sup_{0<\eps<1}\eps^{\nu}||\bar{\poiF}^\eps||_{C^4([0,H^*])}<\infty$, using the explicit expression for the solution in \ref{lem:behavetestfunc}, we find that $||f^\eps||_{C^4([0,H^*])}\leq C \eps^{-\nu}$.
\end{proof}

\section{Averaging}\label{sec:avging_approxing}
In this section we address the issue of approximating $(\sgenonqtf (f^\eps \circ \ham))(\hat{\pj}^{\eps}_u\Xeps)$ with  $(\igenH f_H)(\ham(\hat{\pj}^{\eps}_u\Xeps))$. For this purpose, write
\begin{align}
(\sgenonqtf (f^\eps \circ \ham))(\hat{\pj}^{\eps}_u\Xeps)-(\igenH f_H)(\ham(\hat{\pj}^{\eps}_u\Xeps)) \,\,=& \,\,\bigg((\sgenonqtf (f^\eps \circ \ham))(\hat{\pj}^{\eps}_u\Xeps)-(\sgenonqtf (f^\eps \circ \ham))(\pi\hat{\pj}^{\eps}_u\Xeps)\bigg) \notag \\ 
& \qquad + \bigg((\sgenonqtf (f^\eps \circ \ham))(\pi\hat{\pj}^{\eps}_u\Xeps)-(\igenH f_H)(\ham(\hat{\pj}^{\eps}_u\Xeps))\bigg). \label{eq:justsplit}
\end{align}

Note that for $\eta \in P_{\Lambda}$ we have $(\igenH f_H)(\ham(\eta))=(\Aavg (\sgenonqtf(f_H\circ \ham)))([\eta])$. Hence, when working with the second term in the RHS of equation \eqref{eq:justsplit}, we will be concerned with difference of a function and its average. This motivates the rest of this section until lemma \ref{lem:projdiffgoestozero}.

For $\varphi \in C(\bar{\Ssp};\R)$, we want to bound the difference
\begin{equation}\label{eq:stch_avg_to_bound}
\expteps\left| \int_0^{t\wedge \stopt} \left\{ \varphi(\pi \hat{\pj}^{\eps}_s \Xeps)-(\Aavg \varphi)([\pi \hat{\pj}^{\eps}_s \Xeps])\right\}\,ds\right|.
\end{equation}

Recall $\timepr=\tmprval$.
\begin{lem}\label{lem:func_flow_der_req}
Define $\Phi_{\cdot}:C(P_{\Lambda})\to C(P_{\Lambda})$ by
\begin{equation}\label{eq:func_flow_der_req} 
\Phi_{\varphi}(\eta):=\frac{1}{\timepr}\int_0^{\timepr} s\,\varphi(T(s)\eta)\,ds, \qquad \quad \forall \eta \in P_{\Lambda}.
\end{equation}
Then
\begin{equation}\label{eq:func_flow_der_req2}
\frac{d}{dt}\bigg|_{t=0} \Phi_{\varphi}(T(t)\eta)\,\,=\,\,  \varphi(\eta)-(\Aavg \varphi)([\eta]).
\end{equation}
There is a constant $C$ which does not depend on $\varphi$  such that if $\varphi\in C^2(\bar{\Ssp})$ then
\begin{align}\label{eq:usefulPhiphibound}
||\Phi_{\varphi}||_{C^2(\bar{\Ssp})}\leq C ||\varphi||_{C^2(\bar{\Ssp})}.
\end{align}
\end{lem}
\begin{proof}
\eqref{eq:func_flow_der_req2} follows from the fact that for a $\timepr$-periodic continuous function $f:\R \to \R$, 
\begin{align*}
\lim_{t\to 0}\frac1t\bigg(\frac{1}{\timepr}\int_0^{\timepr}s\,f(s+t)ds  -\frac{1}{\timepr}\int_0^{\timepr}s\,f(s)ds\bigg)\,\,=\,\,f(0)-\frac{1}{\timepr}\int_0^{\timepr}f(s)ds.
\end{align*}
\end{proof}

Let $(\zeta.\nabla)\varphi(\eta)$ denote the Gateaux differential  of $\varphi$ evaluated at $\eta$ in the direction $\zeta$.
\begin{lem}\label{lem:careful}
Let $\varphi \in C^2(\bar{\Ssp})$. Let $\Phi_{\varphi}$ be defined as above. Then $\Phi_{\varphi}\circ \pi \in \dom(\genonqtfext)$, and we have 
\begin{align}\label{eq:BPhipi}
(\tgenonqtf (\Phi_{\varphi}\circ \pi))(\eta)\,\,=\,\,\varphi(\pi\eta)-(\Aavg \varphi)([\pi\eta]),
\end{align}
\begin{align}\label{eq:LPhipi}
(\sgenonqtf (\Phi_{\varphi}\circ \pi))(\eta)\,&=\,(\sgenonqtf (\Phi_{\varphi}\circ \pi))^{(1)}(\eta)\,+\,(\sgenonqtf (\Phi_{\varphi}\circ \pi))^{(2)}(\eta),  
\end{align}
where
\begin{align*}
(\sgenonqtf (\Phi_{\varphi}\circ \pi))^{(1)}(\eta)\, &= \,
 \frac12F^2(\eta)\frac{1}{\timepr}\int_0^{\timepr}s\,\bigg((T(s)\pih \Ind).\nabla\bigg)^2\varphi(\Phi e^{Bs}\la \eta,\Psi\ra)\,ds,\\ 
(\sgenonqtf (\Phi_{\varphi}\circ \pi))^{(2)}(\eta)\, &= \,\, G(\eta)\frac{1}{\timepr}\int_0^{\timepr}s \bigg((T(s)\pih \Ind).\nabla\bigg)\varphi(\Phi e^{Bs}\la \eta,\Psi\ra)\,ds. 
\end{align*}
There is a constant $C$ independent of $\varphi$ such that\footnote{Here we are restricting our attention to $\eta\in \bar{\Ssp}$.} $||(\sgenonqtf (\Phi_{\varphi}\circ \pi))^{(i)}||_{C(\bar{\Ssp})}\leq C ||\Phi_{\varphi}||_{C^2(\bar{\Ssp})}.$
\end{lem}
\begin{proof}
Note that $T(s)\pi\eta=T(s)\Phi\la \eta,\Psi\ra=\Phi e^{Bs}\la \eta,\Psi\ra.$ Hence
\begin{align}\label{eq:Phiphipi}
(\Phi_{\varphi}\circ \pi)(\eta)=\frac{1}{\timepr}\int_0^{\timepr} s\,\varphi(\Phi e^{Bs}\la \eta,\Psi\ra)\,ds, \qquad \quad \forall \eta \in \C.
\end{align}
Because $\varphi \in C^2(\bar{\Ssp})$, the function $\Phi_{\varphi}\circ \pi$ is a $C^2$ function in arguments $\la \eta, \Psi_i\ra$. This shows that $\Phi_{\varphi}\circ \pi \in \dom(\genonqtfext)$. Now, Using \eqref{eq:transgendef} and \eqref{eq:func_flow_der_req2}
\begin{align*}
(\tgenonqtf (\Phi_{\varphi}\circ \pi))(\eta)\,\,&=\,\,\frac{d}{dt}\bigg|_{t=0} (\Phi_{\varphi}\circ \pi) (T(t)\eta)\,\,\\
&=\,\,\frac{d}{dt}\bigg|_{t=0} \Phi_{\varphi}(T(t)\pi\eta)\,\,=\,\,\varphi(\pi\eta)-(\Aavg \varphi)([\pi\eta]).
\end{align*}
The proof of other part is by direct computations.
\end{proof}

\begin{lem}\label{lem:careful2}
Consider the process $M_t$ defined by
\begin{align}\label{eq:apprxdiffapplySV4QTF}
\eps^2\{\Phi_{\varphi}(\pi \hat{\pj}^{\eps}_{t\wedge \stopt} \Xeps)-\Phi_{\varphi}(\pi \hat{\pj}^{\eps}_{0} \Xeps)\}=&\int_0^{t\wedge \stopt}(\tgenonqtf (\Phi_{\varphi}\circ \pi))(\hat{\pj}^{\eps}_{u} \Xeps)\,du \,\,+\\ \notag
&\,\,\eps^2\int_0^{t\wedge \stopt}(\sgenonqtf (\Phi_{\varphi}\circ \pi))(\hat{\pj}^{\eps}_{u} \Xeps)\,du\,\,+\,\,\eps^2M_{t}.
\end{align}
Then $M_{t}$ is a $\F_t$ martingale with quadratic variation
$$\int_0^{t\wedge \stopt}\left( F(\hat{\pj}^{\eps}_{u} \Xeps)\frac{1}{\timepr}\int_0^{\timepr}s\,\bigg((T(s)\pih \Ind).\nabla\bigg)\varphi(\Phi e^{Bs}\la \pi \hat{\pj}^{\eps}_{u} \Xeps,\Psi\ra)\,ds \,\right)^2\,du.$$
\end{lem}
\begin{proof}
Similar to the proof of proposition \ref{prop:Hmart}. 
\end{proof}

\begin{prop}\label{prop:pre_projdiffgoestozero}
Assume $\varphi \in C^2(\bar{\Ssp})$. Then
$$\expteps\left| \int_0^{t\wedge \stopt} \left\{ \varphi(\pi \hat{\pj}^{\eps}_s \Xeps)-(\Aavg \varphi)([\pi \hat{\pj}^{\eps}_s \Xeps])\right\}\,ds\right|\,\,\leq\,\, C(1+t)\eps^2||\varphi||_{C^2(\bar{\Ssp})}.$$
\end{prop}
\begin{proof}
We start with equation \eqref{eq:apprxdiffapplySV4QTF}. The first term on the RHS, using \eqref{eq:BPhipi}, gives
$$\expteps\left| \int_0^{t\wedge \stopt}(\tgenonqtf (\Phi_{\varphi}\circ \pi))(\hat{\pj}^{\eps}_{u} \Xeps)\,du \right|=\expteps\left| \int_0^{t\wedge \stopt} \left\{ \varphi(\pi \hat{\pj}^{\eps}_s \Xeps)-(\Aavg \varphi)([\pi \hat{\pj}^{\eps}_s \Xeps])\right\}\,ds\right|.$$

Now the second term on RHS of \eqref{eq:apprxdiffapplySV4QTF}. Because $F(\eta)=\sigma$ for all $\eta \in \C$, we have $(\sgenonqtf (\Phi_{\varphi}\circ \pi))^{(1)}(\hat{\pj}^{\eps}_{u} \Xeps)=(\sgenonqtf (\Phi_{\varphi}\circ \pi))^{(1)}(\pi\hat{\pj}^{\eps}_{u} \Xeps)$. Using $||(\sgenonqtf (\Phi_{\varphi}\circ \pi))^{(1)}||_{C(\bar{\Ssp})}\leq C ||\Phi_{\varphi}||_{C^2(\bar{\Ssp})}$ from lemma \ref{lem:careful}, we have that
$$\expteps\left|\int_0^{t\wedge \stopt}(\sgenonqtf (\Phi_{\varphi}\circ \pi))^{(1)}(\hat{\pj}^{\eps}_{u} \Xeps)\,du \right|\leq Ct||\Phi_{\varphi}||_{C^2(\bar{\Ssp})}.$$ 
Also,
\begin{align*}
\expteps & \left|\int_0^{t\wedge \stopt}(\sgenonqtf (\Phi_{\varphi}\circ \pi))^{(2)}(\hat{\pj}^{\eps}_{u} \Xeps)\,du \right|\,\,\,\\
&\leq\,\,\, \expteps\left|\int_0^{t\wedge \stopt}(\sgenonqtf (\Phi_{\varphi}\circ \pi))^{(2)}(\pi\hat{\pj}^{\eps}_{u} \Xeps)\,du \right| \\ 
&\qquad \qquad +\expteps\int_0^{t\wedge \stopt}\left|(\sgenonqtf (\Phi_{\varphi}\circ \pi))^{(2)}(\hat{\pj}^{\eps}_{u} \Xeps)-(\sgenonqtf (\Phi_{\varphi}\circ \pi))^{(2)}(\pi\hat{\pj}^{\eps}_{u} \Xeps) \right|\,du.
\end{align*}
For the first term on the RHS above, we use $||(\sgenonqtf (\Phi_{\varphi}\circ \pi))^{(2)}||_{C(\bar{\Ssp})}\leq C ||\Phi_{\varphi}||_{C^2(\bar{\Ssp})}$. The second term, using the Lipshitz condition \eqref{eq:FLipcond} on $G$, can be bounded above by
$$\expteps \int_0^{t\wedge \stopt}K_G\,||(I-\pi)\hat{\pj}^{\eps}_{u} \Xeps||\, C||\Phi_{\varphi}||_{C^2(\bar{\Ssp})}\,du\,\,\,\leq\,\,C||\Phi_{\varphi}||_{C^2(\bar{\Ssp})}\expteps \int_0^{t\wedge \stopt}||(I-\pi)\hat{\pj}^{\eps}_{u} \Xeps||\,du,$$
which, after making use of \eqref{eq:stablespaceestimate} from proposition \ref{prop:stablenormtozero}, can be bounded by 
$C||\Phi_{\varphi}||_{C^2(\bar{\Ssp})}(t\eps+\eps^2)$.

Using the Burkholder-Davis-Gundy inequality (theorem 3.3.28 in \cite{KaratzasShreeve}) for estimating $\expteps|M_{t}|$, and then using \eqref{eq:usefulPhiphibound}, we have the desired result. 
\end{proof}

\begin{lem}\label{lem:projdiffgoestozero}
$$\lim_{\eps \to 0}\expteps \left| \int_0^{t\wedge \stopt} \left\{ (\sgenonqtf (f^\eps \circ \ham))(\pi \hat{\pj}^{\eps}_s \Xeps)-(\igen_H f_H)(\ham(\hat{\pj}^{\eps}_s \Xeps))\right\}\,ds \right|=0.$$
\end{lem}
\begin{proof}
By lemma \ref{lem:fHapprox}, $\sup_{\eps>0}\eps^\nu ||f^\eps \circ \ham||_{C^4(\bar{\Ssp})}<\infty$. And $G$ is a $C^2$ function. Hence $||\sgenonqtf (f^\eps \circ \ham)||_{C^2(\bar{\Ssp})}\leq C\eps^{-\nu}$. Applying proposition \ref{prop:pre_projdiffgoestozero} to $(\sgenonqtf (f^\eps \circ \ham))\circ \pi$ we have
$$\expteps\left| \int_0^{t\wedge \stopt} \left\{(\sgenonqtf (f^\eps \circ \ham))(\pi \hat{\pj}^{\eps}_s \Xeps)-(\Aavg (\sgenonqtf (f^\eps \circ \ham)))([\pi \hat{\pj}^{\eps}_s \Xeps])\right\}\,ds\right| \leq \eps^{{2}} C(1+t)\eps^{-\nu}.$$
Noting that, for $\eta \in \Ssp$ with $\ham(\eta)=\hlev$,
$$(\Aavg (\sgenonqtf (f^\eps \circ \ham)))([\eta])=(\igenH f^\eps)(\hlev),$$
we have
$$\expteps\left| \int_0^{t\wedge \stopt} \left\{(\sgenonqtf (f^\eps \circ \ham))(\pi \hat{\pj}^{\eps}_s \Xeps)-(\igen_H f^\eps)(\ham(\hat{\pj}^{\eps}_s \Xeps))\right\}\,ds\right| \leq  C(1+t)\eps^{{2}-\nu}.$$
Now using part 2 of lemma \ref{lem:fHapprox} and the fact that $\nu$ was chosen to be less than one (see statement before lemma \ref{lem:fHapprox}) gives the desired result. 
\end{proof}

\begin{lem}\label{lem:diffgoestozero}
$$\lim_{\eps \to 0}\expteps \left| \int_0^{t\wedge \stopt} \left\{ (\sgenonqtf (f^\eps \circ \ham))(\hat{\pj}^{\eps}_s \Xeps)-(\sgenonqtf (f^\eps \circ \ham))(\pi \hat{\pj}^{\eps}_s \Xeps)\right\}\,ds \right|=0.$$
\end{lem}
\begin{proof}
Recall from proposition \ref{prop:hofHmart} that $(\sgenonqtf (f^\eps \circ \ham))(\eta)=(\sgenonqtf (f^\eps \circ \ham))^{(1)}(\eta)+(\sgenonqtf (f^\eps \circ \ham))^{(2)}(\eta)$ where
\begin{align*}
(\sgenonqtf (f^\eps \circ \ham))^{(1)}(\eta)&=\frac12 F^2(\eta)\bigg( \dot{f}^{\eps}\big|_{\ham(\eta)}\Psiz^*\Psiz \,\,+\,\,\ddot{f}^{\eps}\big|_{\ham(\eta)}\left( \Psiz^*\la \eta,\Psi\ra\right)^2 \bigg), \\
(\sgenonqtf (f^\eps \circ \ham))^{(2)}(\eta)&=G(\eta)\dot{f}^{\eps}\big|_{\ham(\eta)}\Psiz^*\la \eta,\Psi\ra.
\end{align*}
As noted before $(\sgenonqtf (f^\eps \circ \ham))^{(1)}(\eta)=(\sgenonqtf (f^\eps \circ \ham))^{(1)}(\pi \,\eta)$ because $F(\eta)=\sigma$ for all $\eta \in \C$.

Using the Lipshitz condition \eqref{eq:FLipcond} and that  $|\Psiz^*\la \eta,\Psi\ra|\leq \sqrt{2H^*}\sqrt{\Psiz^*\Psiz}$ for $0\leq s\leq t\wedge \stopt$,
we have
$$|(\sgenonqtf (f^\eps \circ \ham))^{(2)}(\hat{\pj}^{\eps}_s \Xeps)-(\sgenonqtf (f^\eps \circ \ham))^{(2)}(\pi \hat{\pj}^{\eps}_s \Xeps)|\,\,\leq\,\, C\,||(I-\pi)\hat{\pj}^{\eps}_s \Xeps||\,\,|\dot{f}^{\eps}(\ham(\eta))|.$$
Using $\sup_{\eps>0}\eps^{\nu/4} ||f^\eps \circ \ham||_{C^1(\bar{\Ssp})}<\infty$, it is enough to show that, for some $\nu<1$,
$$\lim_{\eps \to 0}\eps^{-\nu/4}\expteps  \int_0^{t\wedge \stopt} ||(I-\pi)\hat{\pj}^{\eps}_s \Xeps||\,ds =0.$$
Application of proposition \ref{prop:stablenormtozero} yields the desired result.
\end{proof}

\begin{prop}\label{prop:diffgoestozero_properapprx}
$$\lim_{\eps \to 0}\expt \left| \int_0^{t\wedge \stopt} \left\{ (\sgenonqtf (f^\eps \circ \ham))(\hat{\pj}^{\eps}_s \Xeps)-(\igen_H f_H)(\ham(\hat{\pj}^{\eps}_s \Xeps))\right\}\,ds \right|=0.$$
\end{prop}
\begin{proof}
Combine lemmas \ref{lem:projdiffgoestozero} and \ref{lem:diffgoestozero}.
\end{proof}


\section{Proof of proposition \ref{prop:mainresultaddnoise}}\label{sec:mainresproof}
Following the remark \ref{rmk:mainproofmotiv} we first prove the tightness of the sequence of laws of $\hprc^\eps$.

\begin{prop}\label{prop:tightness}
There exists a constant $0<C<\infty$ (independent of $\eps$) such that, for any $0\leq t_1\leq t_2\leq T$,
$$\expt  \,\, |\hprc^\eps(t_2\wedge \stopt)-\hprc^\eps(t_1\wedge \stopt)|^4\,\,\,\leq \,\,\,C|t_2-t_1|^2. $$
Thus the laws of $\hprc^\eps$ are tight (see theorem 12.3 of \cite{BillingsleyBook}). 
\end{prop}

\begin{proof}
Recall that $\ham(\hat{\pj}^{\eps}_t\Xeps)=\hprc^\eps(t)$. We have from proposition \ref{prop:Hmart}
$$\ham(\hat{\pj}^{\eps}_{t\wedge \stopt}\Xeps)=\ham(\hat{\pj}^{\eps}_0\Xeps)+\int_0^{t\wedge \stopt}(\sgenonqtf \,\ham)(\hat{\pj}^{\eps}_u \Xeps)du + M_t^{\eps},$$
where $M_t^{\eps}$ is a $\F_t$ martingale with quadratic variation given by 
$$\la M^\eps \ra_t=\int_0^{t\wedge \stopt} Q(\hat{\pj}^{\eps}_s \Xeps)ds,  \qquad \quad Q(\eta):=\left(F(\eta)\Psiz^*\la \eta,\Psi\ra\right)^2.$$
Write $(\sgenonqtf \,\ham)(\eta)=(\sgenonqtf \,\ham)^{(1)}(\eta)+(\sgenonqtf \,\ham)^{(2)}(\eta)$ where
\begin{align*}
(\sgenonqtf  \ham)^{(1)}(\eta)=\frac12 F^2(\eta)\Psiz^*\Psiz, \qquad \quad 
(\sgenonqtf \ham)^{(2)}(\eta)=G(\eta)\Psiz^*\la \eta,\Psi\ra.
\end{align*}

Note that $\ham(\eta)=\ham(\pi\,\eta)$. Also, $Q(\eta)=Q(\pi\,\eta)$ and $(\sgenonqtf \,\ham)^{(1)}(\eta)=(\sgenonqtf \,\ham)^{(1)}(\pi\,\eta)$ (because $F(\eta)=\sigma$ $\forall \eta$). There exists constants $C_1,C_2,C_3,C_4<\infty$ such that
\begin{align*}
&C_1=\sup_{\eta \in \bar{\Ssp}}|(\sgenonqtf \,\ham)^{(1)}(\eta)|, \qquad \quad C_2=\sup_{\eta \in \bar{\Ssp}}|Q(\eta)|, \\
&C_3=\sup_{\eta \in \bar{\Ssp}}|G(\eta)\Psiz^*\la \eta,\Psi\ra|, \qquad \quad  C_4=\sup_{\eta \in \bar{\Ssp}}|\Psiz^*\la \eta,\Psi\ra|, \qquad \quad C_5=\sup_{\eta \in \bar{\Ssp}}|G(\eta)|.
\end{align*}
For $0\leq t\leq \stopt$,
\begin{align*}
|(\sgenonqtf \ham)^{(2)}(\eta)| & \leq |G(\pi\eta)\Psiz^*\la \eta,\Psi\ra|+|G(\eta)-G(\pi\eta)|.|\Psiz^*\la \eta,\Psi\ra| \\
&\leq C_3+C_4 K_G||(I-\pi)\eta||.
\end{align*}
 
Now, using Minkowski's inequality
\begin{align}
\frac{1}{64}\expt  \,\, |\hprc^\eps(t_2\wedge \stopt)-\hprc^\eps(t_1\wedge \stopt)|^4\,\,\,\leq &\,\,\,\expt  \,\,\bigg|\int_{t_1\wedge \stopt}^{t_2\wedge \stopt}(\sgenonqtf \,\ham)^{(1)}(\hat{\pj}^{\eps}_u \Xeps)du \bigg|^4 \label{eq:prftight_t1} \\
& +\expt  \,\,\bigg|\int_{t_1\wedge \stopt}^{t_2\wedge \stopt}(\sgenonqtf \,\ham)^{(2)}(\hat{\pj}^{\eps}_u \Xeps)du \bigg|^4 \label{eq:prftight_t2} \\
& +\expt  \,\,\bigg|M^\eps_{t_1}-M^\eps_{t_2}\bigg|^4. \label{eq:prftight_t3}
\end{align}
For the term in \eqref{eq:prftight_t3}, using the martingale moments inequality (see proposition 3.3.26 and remark 3.3.27 of \cite{KaratzasShreeve}), there exists a constant $C_{m}$ such that
$$\expt  \,\,\bigg|M^\eps_{t_1}-M^\eps_{t_2}\bigg|^4 \,\,\leq\,\,C_{m}\expt  \,\,\bigg|\la M^\eps \ra_{t_1}-\la M^\eps \ra_{t_2}\bigg|^2 \leq C_{m}C_2^2|t_2-t_1|^2.$$
For the term on the RHS of \eqref{eq:prftight_t1},
$$\expt  \,\,\bigg|\int_{t_1\wedge \stopt}^{t_2\wedge \stopt}(\sgenonqtf \,\ham)^{(1)}(\hat{\pj}^{\eps}_u \Xeps)du \bigg|^4 \leq \expt  \,\,\bigg|\int_{t_1\wedge \stopt}^{t_2\wedge \stopt}|(\sgenonqtf \,\ham)^{(1)}(\pi\hat{\pj}^{\eps}_u \Xeps)|\,du \bigg|^4 \leq C_1^4|t_2-t_1|^4.$$
For the term in \eqref{eq:prftight_t2},
\begin{align*}
\expt  \,\,\bigg|\int_{t_1\wedge \stopt}^{t_2\wedge \stopt}(\sgenonqtf \,\ham)^{(2)}(\hat{\pj}^{\eps}_u \Xeps)du \bigg|^4 &  \leq \expt  \,\,\bigg|C_3|t_2-t_1|+C_4 K_G\int_{t_1\wedge \stopt}^{t_2\wedge \stopt}||(I-\pi)\hat{\pj}^{\eps}_u \Xeps||\,du \bigg|^4 \\
& \leq 8 C_3^4|t_2-t_1|^4 + 8 C_4^4K_G^4 \expt  \,\,\bigg|\int_{t_1\wedge \stopt}^{t_2\wedge \stopt}||(I-\pi)\hat{\pj}^{\eps}_s \Xeps||\,ds \bigg|^4.
\end{align*}
The tightness now follows from the following lemma \ref{lem:lemreq4tightness}.
\end{proof}

\begin{lem}\label{lem:lemreq4tightness}
There exists $\eps_0>0$ and a constant $C$ independent of $\eps$, $t_1$ and $t_2$ such that, $\forall \eps \leq \eps_0$ and $0\leq t_1\leq t_2 \leq T$
\begin{align}\label{eq:prftight_req}
\expt  \,\,\bigg|\int_{t_1\wedge \stopt}^{t_2\wedge \stopt}||(I-\pi)\hat{\pj}^{\eps}_s \Xeps||\,ds \bigg|^4 \leq C|t_2-t_1|^4.
\end{align} 
\end{lem}
\begin{proof}
By Holder's inequality, we have
\begin{align*}
\expt  \,\,\bigg|\int_{t_1\wedge \stopt}^{t_2\wedge \stopt}||(I-\pi)\hat{\pj}^{\eps}_s \Xeps||\,ds \bigg|^4 &\leq |t_2-t_1|^3 \,\,\expt \int_{t_1\wedge \stopt}^{t_2\wedge \stopt}||(I-\pi)\hat{\pj}^{\eps}_s \Xeps||^4\,ds \\
& \leq |t_2-t_1|^4 \,\,\expt \sup_{t\in[0,T\wedge \stopt]}||(I-\pi)\hat{\pj}^{\eps}_t \Xeps||^4.
\end{align*}
We will show that there exists $\eps_0>0$ and a constant $C$ independent of $\eps$ such that, $\forall \eps \leq \eps_0$, $\expt \sup_{t\in[0,T\wedge \stopt]}||(I-\pi)\hat{\pj}^{\eps}_t \Xeps||^4 \leq C.$

Using variation-of-constants formula and Minkowski's inequality
\begin{align}
\frac{1}{64} \expt & \sup_{t\in[0,T\wedge \stopt]}  ||(I-\pi)\hat{\pj}^{\eps}_t \Xeps||^4 \,\,\,\notag \\
& \leq\,\,\, \expt \sup_{t\in[0,T\wedge \stopt]}||\Th(t/\eps^2)(I-\pi)\hat{\pj}^{\eps}_0 \Xeps||^4  \label{eq:tightsublemm_use1} \\
& \qquad \qquad + \expt \sup_{t\in[0,T\wedge \stopt]}\left|\left|\int_0^t \Th(\frac{t-u}{\eps^2})(I-\pih)\Ind G(\hat{\pj}^{\eps}_u \Xeps)\,du \right|\right|^4 \label{eq:tightsublemm_use2} \\
& \qquad \qquad + \expt \sup_{t\in[0,T\wedge \stopt]}\left|\left|\int_0^t \Th(\frac{t-u}{\eps^2})(I-\pih)\Ind \sigma\, dW(u) \right|\right|^4. \label{eq:tightsublemm_use3}
\end{align}
The term in \eqref{eq:tightsublemm_use1} is bounded above by $||(I-\pi)\hat{\pj}^{\eps}_0 \Xeps||^4$. The term in \eqref{eq:tightsublemm_use2} is bounded above by
\begin{align*}
8 \,\expt \sup_{t\in[0,T\wedge \stopt]} & \left(\int_0^t \left|\left|\Th(\frac{t-u}{\eps^2})(I-\pih)\Ind \right|\right| |G(\pi\hat{\pj}^{\eps}_u \Xeps)|\,du \right)^4 \\ 
& + \,\,\,8 K_G^4\,\expt \sup_{t\in[0,T\wedge \stopt]}\left(\int_0^t \left|\left|\Th(\frac{t-u}{\eps^2})(I-\pih)\Ind \right|\right| \, ||(I-\pi)\hat{\pj}^{\eps}_u \Xeps||\,du \right)^4 
\end{align*}
which can be bounded above by
\begin{align*}
 8 \,\expt &\sup_{t\in[0,T\wedge \stopt]}   \left(K\int_0^t e^{-\kappa(t-u)/\eps^2} |G(\pi\hat{\pj}^{\eps}_u \Xeps)|\,du \right)^4 \\
& \qquad \qquad \qquad + \,\,\,8 K_G^4\,\expt \sup_{t\in[0,T\wedge \stopt]}\left(K\int_0^t e^{-\kappa(t-u)/\eps^2} \, ||(I-\pi)\hat{\pj}^{\eps}_u \Xeps||\,du \right)^4 \\
& \leq 8 \left(K\int_0^t e^{-\kappa(t-u)/\eps^2} du \right)^4 \expt \sup_{t\in[0,T\wedge \stopt]} \left(|G(\pi\hat{\pj}^{\eps}_t \Xeps)|^4 + K_G^4\,||(I-\pi)\hat{\pj}^{\eps}_t \Xeps||^4 \right) \\
& \leq 8 C_5^4 (K\eps^2/\kappa)^4 + 8 K_G^4(K\eps^2/\kappa)^4 \expt \sup_{t\in[0,T\wedge \stopt]} ||(I-\pi)\hat{\pj}^{\eps}_t \Xeps||^4.
\end{align*}
To deal with the term in \eqref{eq:tightsublemm_use3}, define $g(s):=(\Th(s)(I-\pih)\Ind)(-r)$. Then $g(\cdot)$ is discontinuous at $s=r$. Further,
\begin{align*}
|g'(s+r)|\,\,\,&=\,\,\,\left|\frac{d}{ds}\left((\Th(s)(I-\pih)\Ind)(0)\right)\right|\,\,\,=\,\,\,\left|L_0(\Th(s)(I-\pih)\Ind) \right| \\
&\leq\,\,\, ||L_0||\,||\Th(s)(I-\pih)\Ind|| \,\,\,\leq\,\,\, ||L_0||\,Ke^{-\kappa s}.
\end{align*}
The term in \eqref{eq:tightsublemm_use3} can be written as
\begin{align*}
\expt \sup_{t\in[0,T\wedge \stopt]}\sup_{\theta \in [-r,0]}\left|\int_0^t g\left(\frac{t-u}{\eps^2}+r+\theta \right) \sigma\, dW(u) \right|^4.
\end{align*}
Writing the integral $\int_0^t = \int_0^{(t+\eps^2\theta)\vee 0}+\int_{(t+\eps^2\theta)\vee 0}^t$, doing integration by parts and using Minkowski's inequality, the above term can be bounded by
\begin{align*}
 2^9\,  \expt \sup_{t\in[0,T\wedge \stopt]}\sup_{\theta \in [-r,0]}\bigg[ &|g(r+)|^4\,|\sigma W((t+\eps^2\theta)\vee 0)|^4 \\
& \quad + \left(\frac{1}{\eps^2}\int_0^{(t+\eps^2\theta)\vee 0}|g'(\frac{t-u}{\eps^2}+r+\theta)|\,|\sigma W(u)|\,du\right)^4 \\
&  \quad + \,\,|g(r+\theta)|^4\,|\sigma W(t)|^4 + |g(r-)|^4\,|\sigma W((t+\eps^2\theta)\vee 0)|^4 \\
& \quad + \left(\frac{1}{\eps^2}\int_{(t+\eps^2\theta)\vee 0}^t |g'(\frac{t-u}{\eps^2}+r+\theta)|\,|\sigma W(u)|\,du\right)^4 \,\,\bigg]
\end{align*}
which, using the property $|g'(s+r)|\leq ||L_0||\,Ke^{-\kappa s}$, can be bounded above by
\begin{align*}
2^9\,  \,
\bigg[ &\sup_{t \in [0,2r]}  \sup_{\theta \in [-r,0]} 
\left(
\frac{||L_0||K}{\eps^2}\int_0^{(t+\eps^2\theta)-} e^{-\kappa(t-u)/\eps^2-\kappa \theta} du \right)^4 \\
& \qquad + (r\sup_{s\in [0,r]}|g'(s)| )^4 + 3\, \left(\sup_{s \in [0,2r]}|g(s)|^4\right) \bigg] \,\,  \expt \sup_{t\in[0,T\wedge \stopt]}|\sigma W(t)|^4  
\end{align*}
which in turn can be bounded by the constant $\tilde{C}$ defined as
\begin{align*}
2^9\,  &\left((||L_0||Ke^{\kappa r}/\kappa)^4+ (r\sup_{s\in [0,r]}|g'(s)| )^4 + 3\, \left(\sup_{s \in [0,2r]}|g(s)|^4\right) \right) \expt \sup_{t\in[0,T]}|\sigma W(t)|^4.
\end{align*} 
Collecting all the bounds we have
\begin{align*}
(\frac{1}{64}-8 K_G^4(K\eps^2/\kappa)^4) & \expt \sup_{t\in[0,T\wedge \stopt]}||(I-\pi)\hat{\pj}^{\eps}_t \Xeps||^4 \\
 &\leq \,\,\,||(I-\pi)\hat{\pj}^{\eps}_0 \Xeps||^4 + 8 C_5^4 (K\eps^2/\kappa)^4 + \tilde{C},
\end{align*}
from which the desired result follows.
\end{proof}

Following the remark \ref{rmk:mainproofmotiv} we now prove \eqref{eq:clustpointsolvesmartprob}.
\begin{prop}
For any $f_H\in \dom(\igenH)$, and any $\F_s$ measurable bounded functional $\Theta_s$ of $\hprc^\eps$ we have
\begin{align}
\lim_{\eps \to 0}\expt \bigg[\bigg(f_H(\ham(\hat{\pj}^{\eps}_{t\wedge\stopt} \Xeps))&-f_H(\ham(\hat{\pj}^{\eps}_{s\wedge\stopt} \Xeps)) \notag \\
&-\int_{s\wedge\stopt}^{t\wedge\stopt}(\igenH f_H)(\ham(\hat{\pj}^{\eps}_u \Xeps))\,du \bigg)\Theta_s(\hprc^\eps)\bigg]=0.  \label{eq:clustpointsolvesmartprob_inprop}
\end{align}
Thus, any cluster point of the sequence of laws of $\hprc^\eps$ solves the martingale problem for $\igenH$.
\end{prop}
\begin{proof}
Given $f_H\in \dom(\igenH)$ obtain $f^\eps$ from lemma \ref{lem:fHapprox}. By proposition \ref{prop:hofHmart} we have
\begin{align}\label{eq:property_M_is_mart_for_fHeps}
M_{t}^{f^\eps,\eps}:=f^\eps\circ \ham(\hat{\pj}^{\eps}_{t\wedge \stopt}\Xeps)-f^\eps\circ \ham(\hat{\pj}^{\eps}_0\Xeps)-\int_0^{t\wedge \stopt}(\sgenonqtf (f^\eps\circ \ham))(\hat{\pj}^{\eps}_u\Xeps)du 
\end{align}
is a $\F_t$ martingale. Hence, for any $\F_s$ measurable bounded functional $\Theta_s$ of $\hprc^\eps$ 
\begin{align}\label{eq:clustpointsolvesmartprob_useful1}
\expt \left[\left(f^\eps(\ham(\hat{\pj}^{\eps}_{t\wedge\stopt} \Xeps))-f^\eps(\ham(\hat{\pj}^{\eps}_{s\wedge\stopt} \Xeps))-\int_{s\wedge\stopt}^{t\wedge\stopt}(\sgenonqtf (f^\eps\circ \ham))(\hat{\pj}^{\eps}_u \Xeps)\,du \right)\Theta_s(\hprc^\eps)\right]=0.
\end{align}
Noting that $f^\eps(\ham(\eta))=(f^\eps\circ\ham)(\pi \eta)$ and using part 1 of lemma \ref{lem:fHapprox} and proposition \ref{prop:diffgoestozero_properapprx} we have the desired result \ref{eq:clustpointsolvesmartprob_inprop}.
\end{proof}

For uniqueness of the solution to martingale problem see theorem 8.1.1 in \cite{Ethier_Kurtz}.

\section{Stronger deterministic perturbations}\label{sec:strongerdetperturb}

In this section, we consider $\R$-valued random process $\Xeps(t)$ satisfying
\begin{align}\label{eq:quadnon_main_in_intgl_form_timesclaechange_drophat}
\Xeps(t)=\begin{cases} \icond(0)+\frac{1}{\eps^2}\int_0^tL_0(\hat{\pj}^{\eps}_u \Xeps)du+\frac{1}{\eps}\int_0^t G_q(\hat{\pj}^{\eps}_u \Xeps)du \\
\qquad  +\int_0^t G(\hat{\pj}^{\eps}_u \Xeps)du +\int_0^t F(\hat{\pj}^{\eps}_u \Xeps)dW(u), \qquad \hfill t\geq 0,  \\ \icond(\eps^{-2} t), \hfill -r\eps^2 \leq t\leq 0,\end{cases} 
\end{align}
with $G_q$ satisfying same assumptions as $G$ and in addition (see remark \ref{rmk:about_quad_perturb})
\begin{align}\label{eq:assumption_on_Gq_zero}
\int_0^{\timepr}G_q(T(s)\eta)\,e^{-Bs}\Psiz \,ds\,\,=\,\,0, \qquad \forall \,\eta \in P_{\Lambda}.
\end{align}

For this case, proposition \ref{prop:stablenormtozero} still holds: estimating terms involving $G_q$ in the same way as is done for $G$, we get instead of equation \eqref{eq:propstableesti_use_for_quad}
\begin{align*}
\expteps  \int_0^{t\wedge \stopt} & ||(I-\pi)\hat{\pj}^{\eps}_s \Xeps||\,ds\,\, \leq \,\,Ct\eps + C\eps^2 + Ct\eps + \eps C \, \int_0^{t\wedge \stopt} \expteps||(I-\pi)\hat{\pj}^{\eps}_u \Xeps||\,du . 
\end{align*}

Statement  \eqref{add:newprop:supbound_befGron}  holds with $C$ replaced by $C/\eps$. (Note that $C$ is independent of $\eps$). Consequently, writing $2(1+C^2/\kappa^2)$ as $\Gamma$ we have the following result which is anologous to proposition \ref{add:newprop:supboundGron}
\begin{prop}\label{add:strongdet:newprop:supboundGron}
Define $\beta^{\eps}_s:=||(I-\pi)\hat{\pj}^{\eps}_s \Xeps||$. Then there exists constants $C,\hat{C}$ and $\hat{\eps}_{(a,\delta)}$ such that, given any $a\in [0,1)$,  for $\eps < \hat{\eps}_{(a,\delta)}$ 
\begin{align}\label{add:strongdet:newprop:supboundGron_statement}
\mbbP\left[\sup_{s\in [0,T\wedge \stopt]}\left(\beta^{\eps}_s -\beta^{\eps}_0(1+\frac{1}{2\eps^2}{C}^2s^2)e^{-\kappa s/\eps^2}\right)\,\,\geq \,\,  \Gamma \eps^a \right] \,\,\,\leq\,\,\,\hat{C}\eps^{-a}\sqrt{\frac{r\eps^\delta}{T}\ln\left(\frac{T}{r\eps^\delta}\right)}, \qquad \quad \delta \in (2a,2). 
\end{align}
\end{prop}

Propositions \ref{prop:Hmart} and \ref{prop:hofHmart} hold with obvious changes: for example, the process $M^{\eps}_t$ defined by
\begin{align*}
M_t^{\htf,\eps}= \htf \circ \ham(\hat{\pj}^{\eps}_{t\wedge \stopt}\Xeps) &- \htf \circ \ham(\hat{\pj}^{\eps}_0\Xeps)-\int_0^{t\wedge \stopt}(\sgenonqtf  (\htf \circ \ham))(\hat{\pj}^{\eps}_u \Xeps)du \\
& -\frac{1}{\eps}\int_0^{t\wedge \stopt}G_q(\hat{\pj}^{\eps}_u \Xeps)\Psiz^*\la \hat{\pj}^{\eps}_u \Xeps,\Psi\ra \htf'\big|_{\ham(\hat{\pj}^{\eps}_u \Xeps)}\,ds
\end{align*}
is a $\F_t$ martingale with quadratic variation given by $\int_0^{t\wedge \stopt} Q(\hat{\pj}^{\eps}_s \Xeps)ds$.

Proposition \ref{prop:pre_projdiffgoestozero} still holds albeit with loss of a power of $\eps$, i.e. for $\varphi \in C^2(\bar{\Ssp})$, 
$$\expteps\left| \int_0^{t\wedge \stopt} \left\{ \varphi(\pi \hat{\pj}^{\eps}_s \Xeps)-(\Aavg \varphi)([\pi \hat{\pj}^{\eps}_s \Xeps])\right\}\,ds\right|\,\,\leq\,\, C(1+t)\eps||\varphi||_{C^2(\bar{\Ssp})}.$$

Lemmas \ref{lem:projdiffgoestozero} and \ref{lem:diffgoestozero} and thus proposition \ref{prop:diffgoestozero_properapprx} holds, but we need to supplement them with analogous results for the term
$$\frac{1}{\eps}\int_0^{t\wedge \stopt}G_q(\hat{\pj}^{\eps}_u \Xeps)\Psiz^*\la \hat{\pj}^{\eps}_u \Xeps,\Psi\ra (f^{\eps})'\big|_{\ham(\hat{\pj}^{\eps}_u \Xeps)}\,ds.$$
This is the purpose of the rest of this section.

Let $\tau,\,\Gqfun,\,\Gqfunwh,\,a^{(n)}_q:\,P_{\Lambda}\to \R$ be defined by
\begin{align}\label{eq:taudef}
\tau({\eta}):=\inf \left\{ t > 0\,:\,\la \Th(t)\eta, \Psi\ra = ||\la \eta,\Psi \ra||_2\left[\begin{array}{c}1\\0\end{array}\right] \right\},
\end{align}
where $||\la \eta,\Psi \ra||_2:=\sqrt{\la \eta,\Psi \ra^*\la \eta,\Psi \ra},$ and
\begin{align}\label{eq:defGqfun}
\Gqfun(\eta):=G_q(\eta)\Psiz^*\la \eta,\Psi\ra, \qquad \Gqfunwh(\eta):=-\int_0^{\tau(\eta)}\Gqfun(\Th(s)\eta)ds,
\end{align}
\begin{align}\label{eq:useful_in_quaddrift1_calc_1}
a^{(1)}_q(\eta)=-\int_0^{\tau({\eta})}(\Th(s)\pi\Ind .\nabla)\Gqfun(\Th(s)\pi\eta)ds,
\end{align}
$$a^{(2)}_{q}(\eta)=-\int_0^{\tau({\eta})}(\Th(s)\pi\Ind .\nabla)^2\Gqfun(\Th(s)\pi\eta)ds.$$
Define $b^{(1)}_{q,H}:[0,H^*]\to \R$ by
\begin{align}\label{eq:quaddrift1}
b^{(1)}_{q,H} \circ \eqvtoH = -\Aavg \bigg(G_q(\cdot)a^{(1)}_q(\cdot) \bigg).
\end{align}

\begin{prop}\label{prop:quadnon_centercenter}
$$\lim_{\eps \to 0}\expt \left| \int_0^{t\wedge \stopt}\left\{\frac{1}{\eps}G_q(\pi\hat{\pj}^{\eps}_u \Xeps)\Psiz^*\la \pi\hat{\pj}^{\eps}_u \Xeps,\Psi\ra \, -\,b^{(1)}_{q,H}(\ham(\hat{\pj}^{\eps}_u \Xeps)) \right\}du\right|=0.$$
\end{prop}
\begin{proof}
Recalling definitions \eqref{eq:defGqfun}, and denoting $\tau_t:=\tau(\pi\hat{\pj}^{\eps}_{t}\Xeps)$, we have
\begin{align}\label{eq:prop_quadnon_centercenter_useful_integralform}
\Gqfunwh(\pi\hat{\pj}^{\eps}_{t+dt} \Xeps)-\Gqfunwh(\pi\hat{\pj}^{\eps}_{t} \Xeps)\,\,=&\,\,-\int_{\tau_t}^{\tau_{t+dt}}\Gqfun(\Th(s)\pi\hat{\pj}^{\eps}_{t+dt}X^\eps)ds \notag \\
& \quad -\int_0^{\tau_{t}}\left[\Gqfun(\Th(s)\pi\hat{\pj}^{\eps}_{t+dt}X^\eps)-\Gqfun(\Th(s)\pi\hat{\pj}^{\eps}_{t}X^\eps)\right]ds.
\end{align}
The second term on the RHS can be shown to be
\begin{align}
\bigg((\frac{1}{\eps}G_1+G_s)|_{\hat{\pj}^{\eps}_{t}X^\eps}dt &+ \sigma dW_t\bigg)a^{(1)}_q({\pi\hat{\pj}^{\eps}_{t}X^\eps}) \,\,\,+\,\,\, \frac12 \sigma^2 dt \,a^{(2)}_q({\pi\hat{\pj}^{\eps}_{t}X^\eps}) \notag \\
&-\frac{1}{\eps^2}dt \int_0^{\tau_{t}}(\gen\Th(s)\pi\hat{\pj}^{\eps}_{t}X^\eps.\nabla)\Gqfun \bigg|_{\Th(s)\pi\hat{\pj}^{\eps}_{t}X^\eps}ds,
\end{align}
where $\gen$ is the generator of the semigroup $\Th(t)$. For the last term in the above equation we can use
$$\int_0^{\tau_{t}}(\gen\Th(s)\pi\hat{\pj}^{\eps}_{t}X^\eps.\nabla)\Gqfun \bigg|_{\Th(s)\pi\hat{\pj}^{\eps}_{t}X^\eps}ds\,\,=\,\,\Gqfun(\Th(\tau_t)\pi\hat{\pj}^{\eps}_{t}X^\eps)-\Gqfun(\pi\hat{\pj}^{\eps}_{t}X^\eps).$$

The equation \eqref{eq:prop_quadnon_centercenter_useful_integralform} can be written in the differential form as
\begin{align}\label{eq:prop_quadnon_centercenter_useful_differentialform}
d\Gqfunwh(\pi\hat{\pj}^{\eps}_{t} \Xeps)\,\,=&\,\,-\Gqfun(\Th(\tau_t)\pi\hat{\pj}^{\eps}_{t}X^\eps)d\tau_t \,\,-\,\,\frac12\frac{d}{dv}\big|_{v=0}\Gqfun(\Th(v)\Th(\tau_t)\pi\hat{\pj}^{\eps}_{t}X^\eps)\la d\tau_t,d\tau_t\ra \notag \\
& \quad -\,\la d\Gqfun(\Th(s)\pi\hat{\pj}^{\eps}_{t} \Xeps)|_{s=\tau_t}, d\tau_t\ra \notag \\
& \quad -\,\,\frac{1}{\eps^2}\Gqfun(\Th(\tau_t)\pi\hat{\pj}^{\eps}_{t}X^\eps)dt\,\,+\,\,\frac{1}{\eps^2}\Gqfun(\pi\hat{\pj}^{\eps}_{t}X^\eps)dt \notag \\
& \quad + \,\bigg((\frac{1}{\eps}G_q+G)|_{\hat{\pj}^{\eps}_{t}X^\eps}dt + \sigma dW_t\bigg)a^{(1)}_q({\pi\hat{\pj}^{\eps}_{t}X^\eps}) \notag \\
& \quad + \,\frac12 \sigma^2 dt \,a^{(2)}_q({\pi\hat{\pj}^{\eps}_{t}X^\eps}). 
\end{align}
In order to proceed further, we need to know the evolution of $\tau_t$ which is discussed in remark \ref{rmk:evol_of_taut}. Define the functions $\beta_i:P_{\Lambda}\to \R$ 
\begin{align*}
\beta_1(\eta)&=\Gqfun(\Th(\tau(\eta))\eta)\frac{1}{\om^2}\frac{\Psiz^*B\la \eta, \Psi\ra}{||\la \eta, \Psi\ra||_2^2}, \\
\beta_2(\eta)&=-\Gqfun(\Th(\tau(\eta))\eta)\frac{1}{\om^2}\frac{\Psiz^*B\la \eta, \Psi\ra\, \Psiz^*\la \eta, \Psi\ra}{||\la \eta, \Psi\ra||_2^4}\sigma^2, \\
\beta_3(\eta)&=-\frac12\frac{d}{dv}\bigg|_{v=0}\Gqfun(\Th(v)\Th(\tau(\eta))\eta)\left(\frac{1}{\om^2}\frac{\Psiz^*B\la \eta, \Psi\ra}{||\la \eta, \Psi\ra||_2^2}\right)^2\sigma^2,\\
\beta_4(\eta)&=\sigma^2 (\Th(\tau(\eta))\pih\Ind.\nabla)\Gqfun(\Th(\tau(\eta))\eta)\frac{1}{\om^2}\frac{\Psiz^*B\la \eta, \Psi\ra}{||\la \eta, \Psi\ra||_2^2},
\end{align*} 
Assume $\lim_{\eta \in P_{\Lambda},\, ||\eta||\to 0}\frac{|G_q(\eta)|}{||\eta||^2}<\infty$. Then all the above are bounded on $\Ssp$. The equation \eqref{eq:prop_quadnon_centercenter_useful_differentialform} can be written as
\begin{align*}
\eps\, d\Gqfunwh(\pi\hat{\pj}^{\eps}_{t}X^\eps) \,\,= \,\,&dt\, \left(\frac{1}{\eps}\Gqfun(\pi\hat{\pj}^{\eps}_{t}X^\eps)+G_q|_{\pi \hat{\pj}^{\eps}_{t}X^\eps}\left(a^{(1)}_q({\pi\hat{\pj}^{\eps}_{t}X^\eps}) + \beta_1(\pi{\hat{\pj}^{\eps}_{t}X^\eps})\right)\right)\\
& \quad + dt\, \left(G_q|_{\hat{\pj}^{\eps}_{t}X^\eps}-G_q|_{\pi \hat{\pj}^{\eps}_{t}X^\eps}\right)\left(a^{(1)}_q({\pi\hat{\pj}^{\eps}_{t}X^\eps})+\beta_1(\pi\hat{\pj}^{\eps}_{t}X^\eps)\right) \\
& \quad +\eps G_2|_{\pi\hat{\pj}^{\eps}_{t}X^\eps}dt\left(a^{(1)}_q({\pi\hat{\pj}^{\eps}_{t}X^\eps})+\beta_1(\pi\hat{\pj}^{\eps}_{t}X^\eps)\right) \\
& \quad + \eps\,dt\, \left(G_2|_{\hat{\pj}^{\eps}_{t}X^\eps}-G_2|_{\pi \hat{\pj}^{\eps}_{t}X^\eps}\right)\left(a^{(1)}_q({\pi\hat{\pj}^{\eps}_{t}X^\eps})+\beta_1(\pi\hat{\pj}^{\eps}_{t}X^\eps)\right) \\
& \quad +\eps \sigma dW_t\left(a^{(1)}_q({\pi\hat{\pj}^{\eps}_{t}X^\eps})+\beta_1(\pi\hat{\pj}^{\eps}_{t}X^\eps)\right) \\
& \quad + \eps\frac12 \sigma^2 dt \,a^{(2)}_q({\pi\hat{\pj}^{\eps}_{t}X^\eps})\,\, + \,\,\eps \sum_{i=2}^4\beta_i(\pi\hat{\pj}^{\eps}_{t}X^\eps)dt
\end{align*}

Using Lipshitz condition and proposition \ref{prop:stablenormtozero} on terms involving $G$, $G_q$; using Burkholder-Davis-Gundy inequality on terms involving $dW$; and because of boundedness of the functions, we have
$$\lim_{\eps \to 0}\expt \left| \int_0^{t\wedge \stopt}\left\{\frac{1}{\eps}\Gqfun(\pi\hat{\pj}^{\eps}_{t}X^\eps)+G_q|_{\pi \hat{\pj}^{\eps}_{t}X^\eps}\left(a^{(1)}_q({\pi\hat{\pj}^{\eps}_{t}X^\eps}) + \beta_1(\pi{\hat{\pj}^{\eps}_{t}X^\eps})\right) \right\}du\right|=0.$$  

Applying proposition \ref{prop:pre_projdiffgoestozero} to $G_q(\cdot)\left(a^{(1)}_q(\cdot) + \beta_1(\cdot)\right)$ and then noting that $\Aavg \bigg(G_q(\cdot)\beta_1(\cdot) \bigg)=0$ due to assumption \eqref{eq:assumption_on_Gq_zero} yields the desired result.
\end{proof}


Define 
\begin{align}\label{eq:defGqfun_sc}
\Gqfun(\eta):=\bigg(G_q(\eta)-G_q(\pi\eta)\bigg)\Psiz^*\la \eta,\Psi\ra, \qquad \Gqfunwh(\eta):=\int_0^{\infty}\Gqfun(\Th(s)\eta)ds.
\end{align}
Note that $\Gqfunwh$ is well-defined: using Lipshitz condition on $G_q$ and part 5 of lemma \ref{lem:collec_ext_res}
\begin{align*}
|\Gqfunwh(\eta)|\,\,\leq\,\, \int_0^{\infty}|\Gqfun(\Th(s)\eta)|ds \,\,&\leq\,\, C\int_0^{\infty}||\Th(s)(I-\pi)\eta||\,\,|\Psiz^*\la \Th(s)\eta,\Psi\ra|\,ds \\
& \leq C\,||(I-\pi)\eta||\,||\pi \eta||.
\end{align*}
Following similar steps as in lemma \ref{lem:lemreq4tightness} we can show that $\expt \sup_{t\in[0,T\wedge \stopt]}||(I-\pi)\hat{\pj}^{\eps}_t \Xeps|| \leq C.$ Also, for $s\in [0,t\wedge \stopt]$, $||\pi \hat{\pj}^{\eps}_s \Xeps||\leq \sqrt{2H^*}$. Hence $\expt |\Gqfunwh(\hat{\pj}^{\eps}_{t \wedge \stopt} \Xeps)|$ is bounded. 

Define $a_q:\C \to \R$ and $b^{(2)}_{q,H}:[0,H^*]\to \R$ by
\begin{align}\label{eq:quaddrift2}
a_q(\eta)=\int_0^{\infty}(\Th(s)\Ind .\nabla)\Gqfun(\Th(s)\eta)ds, \qquad \quad b^{(2)}_{q,H} \circ \eqvtoH = \Aavg \bigg(G_q(\cdot)a_q(\cdot) \bigg).
\end{align}

Writing the evolution equation for $d\Gqfunwh(\hat{\pj}^{\eps}_{t \wedge \stopt} \Xeps)$, noting that 
$$\int_0^{\infty}(\gen\Th(s)\hat{\pj}^{\eps}_{t}X^\eps.\nabla)\Gqfun \bigg|_{\Th(s)\hat{\pj}^{\eps}_{t}X^\eps}ds\,\,=\,\,-\Gqfun(\hat{\pj}^{\eps}_{t}X^\eps),$$
and proceeding in similar manner as in proof of proposition \ref{prop:quadnon_centercenter} we arrive at the following proposition \ref{prop:quadnon_stablecenter}:

\begin{prop}\label{prop:quadnon_stablecenter}
$$\lim_{\eps \to 0}\expt \left| \int_0^{t\wedge \stopt}\left\{\frac{1}{\eps}\bigg(G_q(\hat{\pj}^{\eps}_u \Xeps)-G_q(\pi\hat{\pj}^{\eps}_u \Xeps)\bigg)\Psiz^*\la \hat{\pj}^{\eps}_u \Xeps,\Psi\ra \, -\,b^{(2)}_{q,H}(\ham(\hat{\pj}^{\eps}_u \Xeps)) \right\}du\right|=0.$$
\end{prop}

Combinining propositions \ref{prop:quadnon_centercenter} and \ref{prop:quadnon_stablecenter} we have
\begin{prop}\label{prop:quadnon_allfull}
$$\lim_{\eps \to 0}\expt \left| \int_0^{t\wedge \stopt}\left\{\frac{1}{\eps}\bigg( G_q(\hat{\pj}^{\eps}_u \Xeps)\Psiz^*\la \hat{\pj}^{\eps}_u \Xeps,\Psi\ra \, -\,\bigg(b^{(1)}_{q,H}+b^{(2)}_{q,H}\bigg)(\ham(\hat{\pj}^{\eps}_u \Xeps)) \right\}du\right|=0.$$
\end{prop}

This result supplements the averaging results of section \ref{sec:avging_approxing}. It suggests that, in the limit $\eps \to 0$, $G_q$ would result in two additional drift terms (given by equations \eqref{eq:quaddrift1} and \eqref{eq:quaddrift2}) for the diffusion process limit of $\hprc^{\eps}(t)$. However, we are not able to prove the tightness in this case.

\begin{rmk}
\label{rmk:evol_of_taut} \emph{(evolution of $\tau_t$ used in the proof of proposition \ref{prop:quadnon_centercenter})}.
Let $z=\left[\begin{array}{c}z_1\\z_2\end{array}\right]$ denote the coordinates of $\eta \in P_{\Lambda}$ w.r.t $\Phi$, i.e. $z=\la \eta, \Psi\ra$. Then $\la \Th(s)\eta, \Psi\ra=e^{Bs}z$. Evolution of $z(\hat{\pj}^{\eps}_{t}X^\eps)$ is according to
$$dz=\frac{1}{\eps^2}Bz dt \,+\, \Psiz \left((\frac{1}{\eps}G_1+G_2)\big|_{\hat{\pj}^{\eps}_{t}X^\eps}dt+\sigma dW(t)\right).$$
From the definition of $\tau_t$, we have $\cos(\om \tau_t)=\frac{z_1}{||z||_2}$ and $\sin(\om \tau_t)=\frac{z_2}{||z||_2}$. So,
\begin{align*}
\om\,d\tau_t &=\frac{1}{||z||_2^2}(-z_2dz_1+z_1dz_2) \\
& \qquad +\frac12\left(\frac{2z_1z_2}{||z||_2^4}\la dz_1,dz_1\ra + 2\frac{(z_2^2-z_1^2)}{||z||_2^4}\la dz_1,dz_2\ra + \frac{-2z_1z_2}{||z||_2^4}\la dz_2,dz_2\ra\right).
\end{align*}
Using $-z_2\Psiz_1+z_1\Psiz_2=-\frac{1}{\om}\Psiz^*Bz$, and  $\la dz_i,dz_j\ra=\Psiz_i\Psiz_j\sigma^2dt,$ and
$$z_1z_2\Psiz_1^2+(z_2^2-z_1^2)\Psiz_1\Psiz_2-z_1z_2\Psiz_2^2=(z_2\Psiz_1-z_1\Psiz_2)(z_1\Psiz_1+z_2\Psiz_2)=\frac{1}{\om}\Psiz^*Bz\,\Psiz^*z,$$
we have, with $z=\la \hat{\pj}^{\eps}_{t}X^\eps,\Psi \ra$,
\begin{align*}
d\tau_t=-\frac{1}{\eps^2}dt-\frac{1}{\om^2}\frac{\Psiz^*Bz}{||z||_2^2}\left((\frac{1}{\eps}G_1+G_2)\big|_{\hat{\pj}^{\eps}_{t}X^\eps}dt+\sigma dW(t)\right)+\frac{1}{\om^2}\frac{\Psiz^*Bz\, \Psiz^*z}{||z||_2^4}\sigma^2 dt
\end{align*}
and
$$\la d\tau_t,d\tau_t\ra=\frac{1}{\om^4}\left(\frac{\Psiz^*Bz}{||z||_2^2}\right)^2\sigma^2 dt.$$
\end{rmk}


\section{Example}\label{sec:numsim}
Consider the following equation:
\begin{align}\label{eq:examplesys_quadadded}
dX(t)=-\frac{\pi}{2}X(t-1)dt + \eps \gamma_qX^2(t-1)dt +\eps^2 \gamma_c X^3(t-1)dt + \eps \sigma dW.
\end{align}
In this case $L_0\eta=-\frac{\pi}{2}\eta(-1)$. The characteristic equation $\lambda+\frac{\pi}{2}e^{-\lambda}=0$ has countably infinite roots on the complex plane. The roots with the largest real part are $\pm i\frac{\pi}{2}$. Hence $L_0$ satisfies the assumption \ref{ass:assumptondetsys}. The basis $\Phi$ for $P_{\Lambda}$ and the function $\Psi$ can be evaluated as
$$\Phi(\theta)=[\cos(\frac{\pi}{2}\theta)\,\,\,\sin(\frac{\pi}{2}\theta)], \qquad \quad \Psi(\tau)=N\left[\begin{array}{c}\cos(\frac{\pi}{2}\tau)-\frac{\pi}{2}\sin(\frac{\pi}{2}\tau) \\ \frac{\pi}{2}\cos(\frac{\pi}{2}\tau)+\sin(\frac{\pi}{2}\tau) \end{array}\right],$$
where $N=2/(1+(\pi/2)^2)$. Using $T(s)\cos(\frac{\pi}{2}\cdot)=\cos(\frac{\pi}{2}(t+\cdot)),$
the averaged drift and diffusions can be calculated (see \eqref{eq:avgdr1}, \eqref{eq:avgdr2} and \eqref{eq:avgdiff}) as
$$b_H^{(1)}(\hlev)=N\sigma^2, \qquad b_H^{(2)}(\hlev)=-\gamma_c\frac32\Psiz_2 \hlev^2, \qquad \sigma_H^{2}(\hlev)=2N\sigma^2 \hlev.$$

Now we evaluate $b^{(1)}_{q,H}(\hlev)$. Note that $G_q(\eta)=\gamma_q(\eta(-1))^2$. Let $\Gqfun(\eta):=\gamma_q(\eta(-1))^2\Psiz^*\la \eta,\Psi\ra$ as in \eqref{eq:defGqfun} and $\tau(\eta)$, $a^{(1)}_q$ be defined as in equations \eqref{eq:taudef}, \eqref{eq:useful_in_quaddrift1_calc_1}. Note that 
$$b^{(1)}_{q,H}(\hlev)=-\frac{1}{\timepr}\int_0^{\timepr}G_q(\Th(u)\eta)\,a^{(1)}_q(\Th(u)\eta)\,du, \quad \text{ for } \eta \in P_{\Lambda} \text{ such that } \ham(\eta)=\hlev.$$
To make calculations easy, we select $\eta(\cdot)=\sqrt{2\hlev}\cos(\frac{\pi}{2}\cdot)=\sqrt{2\hlev}\Phi\left[\begin{array}{c}1 \\ 0\end{array}\right] $,
 and for this $\eta$ it can be checked that $\Th(u)\eta=\sqrt{2\hlev}\,\Phi e^{Bu}\left[\begin{array}{c}1 \\ 0\end{array}\right]$, and $\tau(\Th(u)\eta)=\timepr-u$. Using
$$(\xi.\nabla)\Gqfun(\eta)=2\gamma_q\eta(-1)\xi(-1)\Psiz^*\la \eta,\Psi\ra + \gamma_q\eta^2(-1)\Psiz^*\la \xi,\Psi\ra,$$
and $\Psiz^* \la \Th(u)\eta, \Psi \ra = \sqrt{2\hlev}\,e^{Bu}\left[\begin{array}{c}1 \\ 0\end{array}\right]$ and $\Psiz \la \Th(s)\pih\Ind, \Psi \ra=\Psiz^* e^{Bs}\Psiz$ in evaluating $a^{(1)}_q$, we have 
$b^{(1)}_{q,H}(\hlev)=-\gamma_q^2\frac{1}{2\pi}\Psiz_1\Psiz_2(2\hlev)^2$. 

Now we evaluate $b^{(2)}_{q,H}(\hlev)$. Let $\Gqfun(\eta):=\gamma_q((\eta(-1))^2-(\pi\eta(-1))^2)\Psiz^*\la \eta,\Psi\ra$ as in \eqref{eq:defGqfun_sc} and $a_q$ be defined as in equation \eqref{eq:quaddrift2}. Then $b^{(2)}_{q,H}(\hlev)$ equals
\begin{align}\label{eq:bh2evaluseful}
\gamma_q^2\frac{1}{\timepr}\int_0^{\timepr}du \big((\Th(u)\eta)(-1)\big)^2\int_0^{\infty} 2 (\Th(s+u)\eta)\bigg|_{-1}(\Th(s)(I-\pi)\Ind)\bigg|_{-1}\Psiz^*\la \Th(s+u)\eta,\Psi\ra\,ds,
\end{align}
for $\eta \in P_{\Lambda}$ such that $\ham(\eta)=\hlev$. Taking $\eta(\cdot)=\sqrt{2\hlev}\cos(\frac{\pi}{2}\cdot)$ the above can be evaluated \emph{numerically}. Let $x(t),$ $t\in [-1,\infty]$ be the solution for 
$$\begin{cases}
\dot{x}(t)=-\frac{\pi}{2}x(t-1), \quad t>0 \\ x(t)=-\Phi(t)\Psiz, \quad  -1\leq t<0 \\ x(0)=1-\Phi(0)\Psiz, \quad  t=0 \end{cases}.$$
Then  $(\Th(t)(I-\pi)\Ind)(-1)=x(t-1)$. Because of the exponential decay of the norm of $(\Th(s)(I-\pi)\Ind)$, it is enough to evaluate the inner integral in \eqref{eq:bh2evaluseful} for a finite value of $s$ for a good enough approximation. On evaluating, we get $b^{(2)}_{q,H}(\hlev)\approx -0.1973\gamma_q^2(2\hlev)^2$.

The averaged equation corresponding to \eqref{eq:examplesys_quadadded}  is
\begin{align}\label{eq:examplesysavgd_quadadded}
d\hlev(t)=\bigg(N\sigma^2-\frac32\gamma_c\Psiz_2 \hlev^2-\gamma_q^2(\frac{1}{2\pi}\Psiz_1\Psiz_2+0.1973)(2\hlev)^2\bigg)dt + \sqrt{2N\sigma^2 \hlev}\, dW.
\end{align}

Now we illustrate our results employing numerical simulations.

Draw a random sample of $N_{samp}$ particles with $\hlev$ values $\{\hlev^0_i\}_{i=1}^{Nsamp}$. Simulate them according to  \eqref{eq:examplesysavgd_quadadded} for $0\leq t \leq T_{end}$.

Fix $\eps=0.025$. Simulate \eqref{eq:examplesys_quadadded} for $0 \leq t \leq T_{end}/\eps^2$ using initial trajectories $\{\sqrt{2\hlev^0_i} \cos(\om_c\cdot)\}_{i=1}^{Nsamp}$. 

Let $\tau^\eps:=\inf \{t\geq 0: |X(t)|\geq \sqrt{2H^*}\}$ and $\tau^\hlev:=\inf \{t\geq 0: \hlev(t)\geq H^*\}$

We can check whether the following pairs are close.
\begin{enumerate}
\item the distribution of $\ham(\pj_{T_{end}/\eps^2}X)$ from \eqref{eq:examplesys_quadadded} \emph{and} the distribution of $\hlev(T_{end})$ from \eqref{eq:examplesysavgd_quadadded},
\item distribution of $\eps^2 \tau^\eps$  \emph{and} the distribution of $\tau^\hlev$.
\end{enumerate}

We took $H^*=1.5$, $T_{end}=2$, $N_{samp}=4000$, and $\sqrt{2\{\hlev^0_i\}_{i=1}^{Nsamp}}=1.2$. Figures \ref{fig:tcdf} and \ref{fig:taucdf} answer the above questions. Three cases are considered with $\sigma=1$ fixed: $(\gamma_q=0,\gamma_c=0)$, $(\gamma_q=0,\gamma_c=1)$, $(\gamma_q=1/\sqrt{3},\gamma_c=0)$.

\begin{figure}
\centering
\begin{minipage}{.5\textwidth}
  \centering
\includegraphics[scale=0.4]{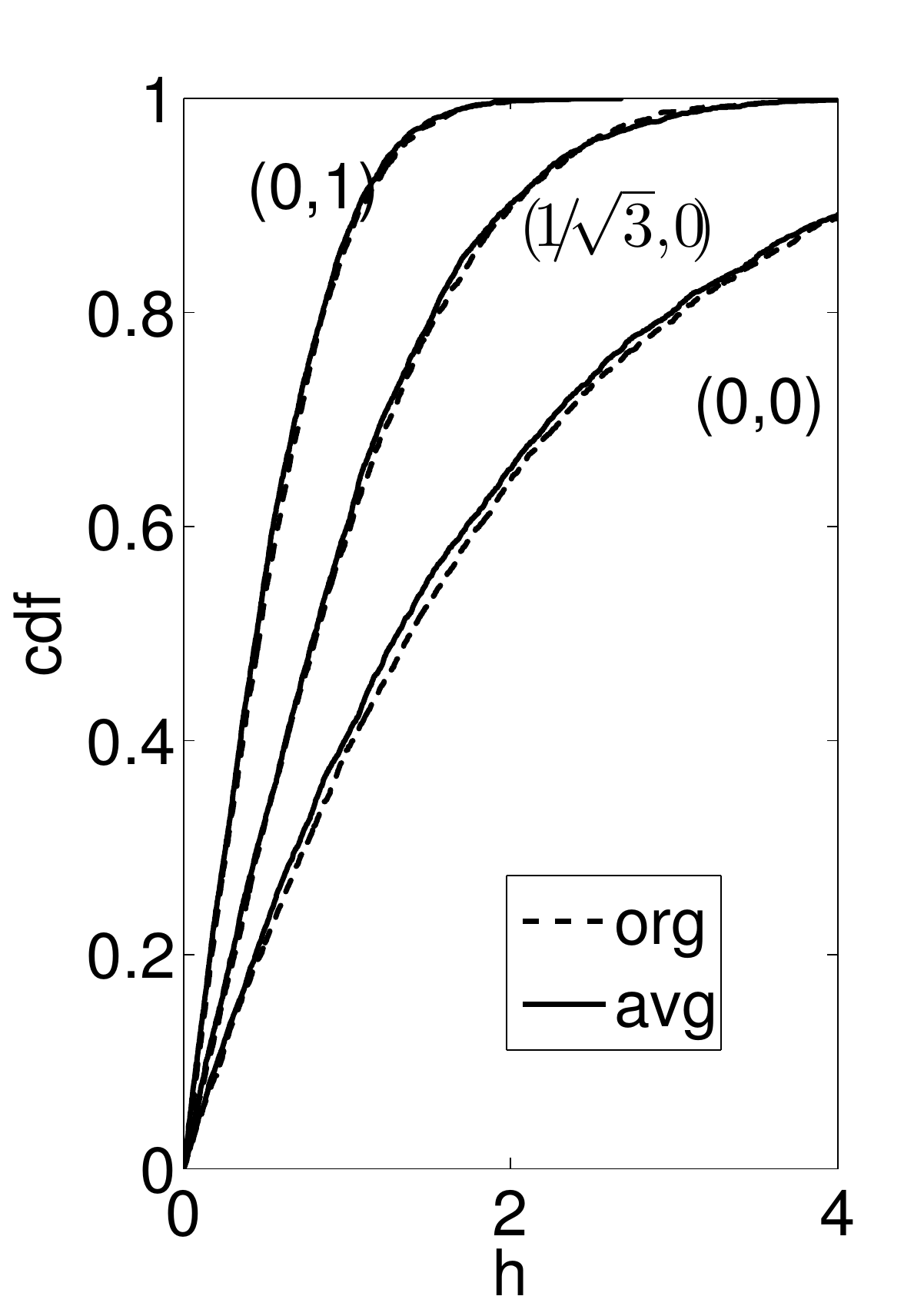}
  \caption{cdf of $\ham(\pj_{2/\eps^2}x)$ ({\tt{org}}) and $\hlev(2)$ ({\tt{avg}}). Numbers in the brackets indicate $(\gamma_q,\gamma_c)$ values.}
  \label{fig:tcdf}
\end{minipage}%
\begin{minipage}{.5\textwidth}
  \centering
  \includegraphics[scale=0.4]{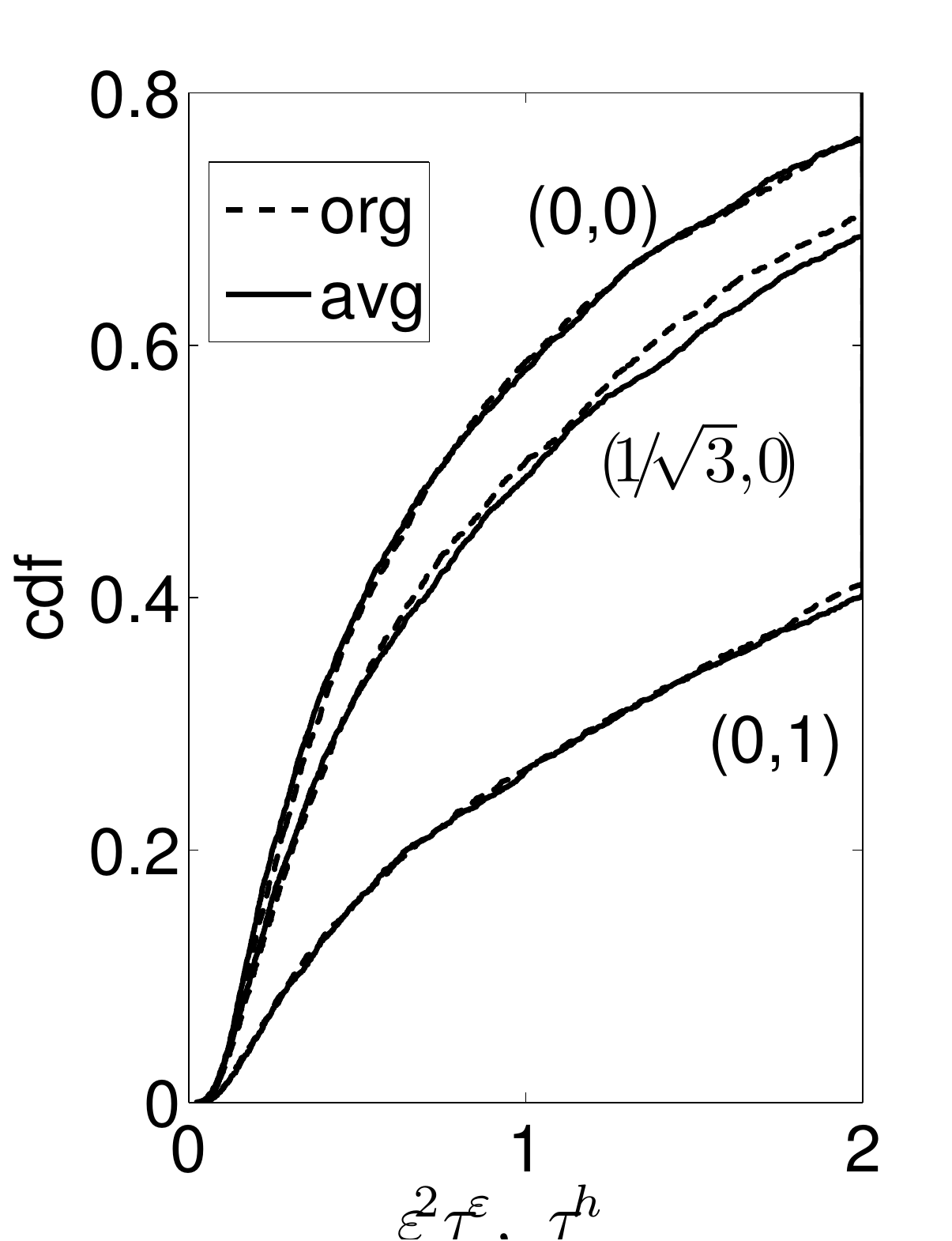}
  \caption{cdf of $\eps^2 \tau^\eps$ ({\tt{org}}) and cdf of $\tau^\hlev$  ({\tt{avg}}). Numbers in the brackets indicate $(\gamma_q,\gamma_c)$ values. The cdf value at $\eps^2\tau^\eps=2$ indicates the fraction of particles whose modulus exceeded $\sqrt{2H^*}$ before the time $2/\eps^2$.}
  \label{fig:taucdf}
\end{minipage}
\end{figure}

\section*{Acknowledgments}
The authors would like to acknowledge the support of the National Science Foundation under grant numbers CMMI 1000906 and 1030144. Any opinions, findings, and conclusions or recommendations expressed in this paper are those of the authors and do not necessarily reflect the views of the National Science Foundation.


\appendix

\section{Multiplicative Noise}\label{sec:multnoise}
This section deals with the case of $F$ depending on $\eta$ (not necessarily a constant).

Let $G:\C\to \R$ be a $C^2$ function satisfying the Lipshitz condition \eqref{eq:FLipcond}
and let $F:\C \to \R$ be a $C^2$ \emph{bounded} function satisfying the Lipshitz condition:
\begin{align}\label{mult:eq:FLipcond}
|F(\eta)-F(\tilde{\eta})| & \leq K_F||\eta-\tilde{\eta}||, \quad \forall \,\eta,\tilde{\eta} \in \C.
\end{align}
It can be shown that there exists a constant $K_g$ such that $F,G$ satisfy the growth condition \eqref{eq:FGgrowthforunique}.
In order to be able to prove a lemma analogous to lemma \ref{lem:fHapprox}, we need to make sure that the averaged diffusion coefficient is not degenerate. For this purpose, we also assume that $F$ satisfies \eqref{mult:eq:Flowboundgrowth}.

From theorem 1.1 of \cite{Reis_Emery_Ineq_voc_for_SDDE}, it can be easily seen that proposition \ref{prop:vocformula} holds when $\sigma$ is replaced by $F(\pj X)$ (with the additional assumptions on $F$ listed above).

First task is to show that the $Q_{\Lambda}$ projection is small. 

In proving lemma \ref{add:newlem:supbound},  the only property of $\sigma W=:Z$ used is that $Z$ is a martingale with quadratic variation bounded on $[0,T]$. This property still holds for $Z^\eps_t=\int_0^tF(\hat{\pj}_u\Xeps) dW(u)$ because $F$ is bounded. Consequently, result analogous to lemma \ref{add:newlem:supbound} holds even for $\Upsilon^{\eps}_s:=\left|\left|\int_0^s\Th(\frac{s-u}{\eps^2})(I-\pih)\Ind dZ^\eps(u)\right|\right|=\left|\left|\int_0^s\Th(\frac{s-u}{\eps^2})(I-\pih)\Ind F(\hat{\pj}_u\Xeps) dW(u)\right|\right|$. Now it is easy to see that proposition \ref{add:newprop:supboundGron} also holds for the present case of $F$ depending on $\eta$.

Fix  $H_*,H^*\in \R^+$ and let $$\Ssp:=\{\eta \in P_{\Lambda}\,:\,H_*<\ham(\eta)<H^*\}.$$
Assume that the initial condition $\hat{\pj}_0\Xeps$ is such that $\pi \hat{\pj}_0\Xeps \in \Ssp$. Define the stopping time
$$\stopt := \inf \{t\geq 0 \,:\, \pi\hat{\pj}^{\eps}_t\Xeps\not\in \Ssp \}.$$

\begin{prop}\label{mult:prop:stablenormtozero}
For any $\nu<1$,
\begin{align}\label{mult:eq:lem:stablenormtozero}
\lim_{\eps \to 0}\eps^{-\nu}\expteps  \int_0^{t\wedge \stopt} ||(I-\pi)\hat{\pj}^{\eps}_s \Xeps||\,ds =0.
\end{align}
\end{prop}
\begin{proof}
The proof is in the same spirit as in the case for $F=\sigma$. Only difference lies in the term
\begin{align}\label{mult:eq:vocnonlinstable}
\expteps  \int_0^{t\wedge \stopt}\left|\left|\int_0^s\Th(\frac{s-u}{\eps^2})(I-\pih)\Ind F(\hat{\pj}^{\eps}_u \Xeps) dW(u)\right|\right|\,ds.
\end{align}
Write $Z^\eps(t)=\int_0^t F(\hat{\pj}^{\eps}_u \Xeps) dW(u)$. Then $Z^\eps$ is a martingale. We try to bound the term
\begin{align}\label{mult:eq:termtobebounded}
\expteps \sup_{\theta \in [-r,0]}\left|\int_0^s \left(\Th(\frac{s-u}{\eps^2})(I-\pih)\Ind\right)(\theta)\,dZ^\eps(u)\right|.
\end{align}
This is done in \eqref{add:newlem:avgbound} for the case of $Z^\eps=\sigma W$. Only properties of $Z^\eps$ used in \eqref{add:newlem:avgbound} are that $Z^\eps$ is a martingale with quadratic variation bounded for finite time. These properties still hold when $dZ^\eps_u=F(\hat{\pj}^{\eps}_u \Xeps) dW(u)$ with bounded $F$. So, there exists a constant $C$ such that $\eqref{mult:eq:termtobebounded}\leq C\eps$ for $s\in [0,T\wedge \stopt]$. Now follow same approach as in proof of \ref{prop:stablenormtozero}.
\end{proof}

The averaged drift coefficient $b_H^{(1)}$ is given as in \eqref{eq:avgdr1}
\begin{align}\label{mult:eq:avgdr1}
b_H^{(1)}(\hlev)&=\frac{1}{\tmpr}\int_0^{\tmpr}\frac12F^2(T(s)\sqrt{2\hlev}\cos\omega_c\cdot)\Psiz^*\Psiz\,ds,
\end{align}
$b_H^{(2)}$ is same as in \eqref{eq:avgdr2} and the averaged diffusion coefficient is given as in \eqref{eq:avgdiff}
\begin{align}\label{mult:eq:avgdiff}
\sigma^2_{H}(\hlev)&=\frac{1}{\tmpr}\int_0^{\tmpr}F^2(\sqrt{2\hlev}T(s)\cos\omega_c\cdot) \left( \Psiz^*\la \sqrt{2\hlev}T(s)\cos\omega_c\cdot,\Psi\ra\right)^2 \,ds. 
\end{align}
We make one further assumption on $F$:
\begin{align}\label{mult:eq:Flowboundgrowth}
\exists \,c:(0,\infty)\to \R^+ \text{ such that } \sigma_H^2(\hlev)>c(H_*) \text{ for all } h\geq H_*. 
\end{align}

Define an operator $\igenH$ by\footnote{$\igenH f_H\in C([H_*,H^*])$ means that $\igenH f_H\in C((H_*,H^*))$ and the limits $\lim_{\hlev \downarrow H_*}\igenH f_H$, $\lim_{\hlev \uparrow H^*}\igenH f_H$ exists and are finite. See chapter 8 section 1 of \cite{Ethier_Kurtz}.}
\begin{align}
\dom(\igenH)=\bigg\{ f_H \in C([H_*,H^*]) & \cap C^2((H_*,H^*))\,:\,\,\igenH f_H\in C([H_*,H^*]) \notag \\
& \text{ and } \lim_{\hlev \uparrow H^*}(\igenH f_H)(\hlev)=0= \lim_{\hlev \downarrow H_*}(\igenH f_H)(\hlev) \bigg\}, \notag \\ 
\text{for } \hlev \in (H_*,H^*), \qquad (\igenH f_H)(\hlev)&=b_H(\hlev)\dot{f}_H(\hlev)+\frac12\sigma^2_{H}(\hlev)\ddot{f}_{H}(\hlev). \label{mult:eq:igenHdef_limitproc}
\end{align}

\begin{thm}\label{mult:prop:mainresultaddnoise}
Under the assumptions on $F$ listed in this section, the statement of theorem \ref{prop:mainresultaddnoise} holds with $\stoptH$ replaced with $\stoptH:=\inf\{t\geq 0\,:\,\avgHproc(t)\geq H^* \text{ or }\avgHproc(t)\leq H_*\}.$
\end{thm}

The proof of above result follows the same strategy stated in remark \ref{rmk:mainproofmotiv}.

Analogous result to the lemma \ref{lem:behavetestfunc} is the following:
\begin{lem}\label{mult:lem:behavetestfunc} 
If $\poiF\in C([H_*,H^*])$, then there exists a solution of 
$$\igenH u=\poiF, \qquad \hlev \in (H_*,H^*)$$ 
such that $u\in C^1([H_*,H^*])$. This solution is unique upto the choice of $u(H_*)$ and $u'(H_*)$. Further, there exists constants $C_i$ independent of $\poiF$ such that
$$||u||_{C([H_*,H^*])}\leq |u(H_*)|+C_1|u'(H_*)|+C_2||\poiF||_{C([H_*,H^*])}.$$
\end{lem}
\begin{proof}
Define $J(\hlev):=\int_{H_*}^h\frac{2b_H(s)}{\sigma_H^2(s)}ds$. Then
$$u(h)=c_1+c_2\int_{H_*}^he^{-J(s)}ds+\int_{H_*}^h\int_{H_*}^se^{-(J(s)-J(r))}\frac{2F(r)}{\sigma_H^2(r)}dr\,ds.$$
Here $c_1=u(H_*)$ and $c_2=u'(H_*)$.
\end{proof}

Result analogous to lemma \ref{lem:fHapprox} can be easily be proved.

Whole of section \ref{sec:avging_approxing} still holds with a few minor changes---for example, in proof of proposition \ref{prop:pre_projdiffgoestozero}, we cannot use $(\sgenonqtf (\Phi_{\varphi}\circ \pi))^{(1)}(\hat{\pj}^{\eps}_{u} \Xeps)=(\sgenonqtf (\Phi_{\varphi}\circ \pi))^{(1)}(\pi\hat{\pj}^{\eps}_{u} \Xeps)$ anymore; we have to estimate the term involving $(\sgenonqtf (\Phi_{\varphi}\circ \pi))^{(1)}$ in the same way as we did for $(\sgenonqtf (\Phi_{\varphi}\circ \pi))^{(2)}$. Similar change must be made in the proof of proposition \ref{prop:tightness}.

The following result is required for the proof of lemma \ref{lem:lemreq4tightness}. 
\begin{lem}\label{mult:lem:lemreq4tightness}
There exists $\eps_0>0$ and a constant $C$ independent of $\eps$ such that, $\forall \eps \leq \eps_0$, $\expt \sup_{t\in[0,T\wedge \stopt]}||(I-\pi)\hat{\pj}^{\eps}_t \Xeps||^4 \leq C.$
\end{lem}

\begin{proof}
Same as the proof of lemma \ref{lem:lemreq4tightness} except for $\sigma W$ replaced by $Z^\eps(t)=\int_0^t F(\hat{\pj}^{\eps}_u \Xeps) dW(u)$. Note that $\expt \sup_{t\in[0,T]}|Z^\eps(t)|^4 \,\,\leq \,\,C(\max |F|)^2T^2$ by Burkholder-Davis-Gundy inequality. 
\end{proof}
This completes the proof of theorem \ref{mult:prop:mainresultaddnoise}.

Now we give an example. Consider
\begin{align}\label{mult:eq:examplesys}
dX(t)=L_0(\pj_t X)dt + \eps L_1(\pj_tX) dW,
\end{align}
where $L_i$ are bounded linear operators, with deterministic system satisfying assumption \ref{ass:assumptondetsys}. Note that $|F(\eta)|=|L_1(\eta)|$ is not bounded and hence equation \eqref{mult:eq:examplesys} does not satisfy the hypothesis. Nevertheless we discuss this example because numerical simulations seem to show close agreement with the results obtained above.  

Let $||\Psiz||^2:=\Psiz^*\Psiz$ and $||L_1\Phi||^2:=(L_1\Phi_1)^2+(L_1\Phi_2)^2$. 
The averaged equation corresponding to \eqref{mult:eq:examplesys} is
\begin{align}\label{mult:eq:examplesysavgd}
d\hlev(t)\,=b_H(\hlev) dt\,+\,\sigma_H(\hlev)\,dW,
\end{align}
where
\begin{align*}
b_H(\hlev)=\frac12||\Psiz||^2||L_1\Phi||^2\hlev, \qquad 
\sigma_H^2(\hlev)&=\left(\frac12||\Psiz||^2||L_1\Phi||^2+(\Psiz^*L_1\Phi)^2\right)\hlev^2.
\end{align*}
The Lyapunov exponent for \eqref{mult:eq:examplesysavgd} can be calculated to be 
\begin{align}\label{mult:eq:lyapcalc}
\lambda_{avg}&=\left(\frac12||\Psiz||^2||L_1\Phi||^2\right)-\frac12\left(\frac12||\Psiz||^2||L_1\Phi||^2+(\Psiz^*L_1\Phi)^2\right)\\
&=\frac14||\Psiz||^2||L_1\Phi||^2\left(1-2\left(\frac{\Psiz^*L_1\Phi}{||\Psiz||\,||L_1\Phi||}\right)^2\right). \notag
\end{align}

Define $\lambda^\eps(t):=\frac1t \log\, \sup_{s\in[t,t+nr]}|X(s)|$ with $n\in \mathbb{N}$ such that $nr\geq \frac{2\pi}{\om_c}$ (here $n$ is chosen so as to avoid oscillations in the modulus of $X$). It can be checked that for large $t$, $\lambda^\eps(t)$ is close to $\eps^2\frac12\lambda_{avg}$. The $\frac12$ arises from the fact that $\hlev$ is quadratic in $X$.

We took $L_0\eta=-\frac{\pi}{2}\eta(-1)$ and $L_1\eta=\eta(-1)$. The Lyapunov exponent for \eqref{mult:eq:examplesysavgd} can be calculated to be $\lambda_{avg} \approx -0.122$.
Five realizations of trajectories of \eqref{mult:eq:examplesys} are simulated with $\eps=0.1$; and in the figure \ref{mult:fig:linearlyapfull} we show mean, min and max (of the five trajectories) for $\lambda^\eps(t):=\frac1t \log\, \sup_{s\in[t-r,t]}|X(s)|$. For $t$ large $\lambda^\eps(t)$ is close to $-0.0005$ and we have $\eps^2\frac12\lambda_{avg} \approx -0.0006$.
\begin{figure}
\centering
  \includegraphics[scale=0.6]{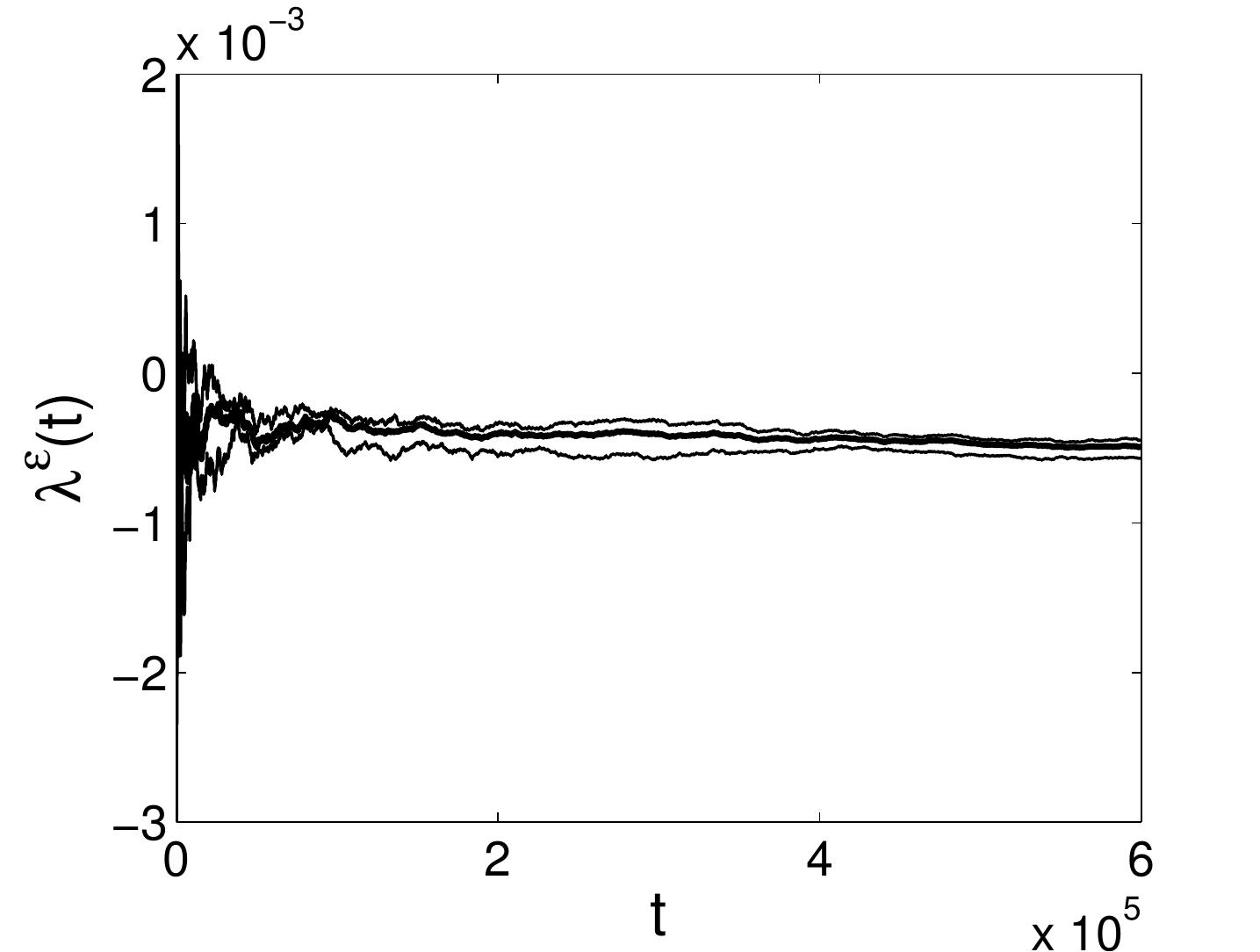}
  \caption{mean, min and max (of the five trajectories) for $\lambda^\eps(t):=\frac1t \log\, \sup_{s\in[t-r,t]}|X(s)|$}
  \label{mult:fig:linearlyapfull}
\end{figure}

The stability condition $\lambda_{avg}<0$ translates to 
\begin{align}\label{mult:eq:awsomeresultforlinear}
\left|\frac{\Psiz^*L_1\Phi}{||\Psiz||\,||L_1\Phi||}\right|>\frac{1}{\sqrt{2}}
\end{align}
leading to the interpretation that \emph{the alignment of the vectors $\Psiz$ and $L_1\Phi$ is a measure of the stability of the system.}

We do not claim that $\eps^2\frac12\lambda_{avg}$ is the maximal exponential growth rate of \eqref{mult:eq:examplesys}. The result that we proved concerns with weak convergence and hence we cannot comment on the almost sure properties of the trajectories.

 Further, we are restricting to systems satisfying assumption \ref{ass:assumptondetsys}. \cite{SEAM_AOP_2} discusses methods to obtain bounds on the maximal exponential growth rates of more general class of delay equations. However the bounds given in \cite{SEAM_AOP_2} are not optimal for systems satisfying assumption \ref{ass:assumptondetsys}. For example, consider
\begin{align}\label{mult:eq:tocomparewithSEAM}
dX(t)=-\frac{\pi}{2}X(t-1)dt+\eps\left(\int_{-1}^0 x(t+s)ds\right)dW
\end{align}
and compare with equation $VI$ of \cite{SEAM_AOP_2}. According to theorem 4.1 of \cite{SEAM_AOP_2} the maximal exponential growth rate $\lambda_1$ of \eqref{mult:eq:tocomparewithSEAM} is bounded above by
$$\lambda_1\leq \inf\{\theta(\delta,\alpha)\,:\,\delta\in \R, \,\,\alpha\in \R^+\},$$
where
\begin{align}\label{mult:eq:thetaofSEAM}
\theta(\delta,\alpha)\,:=\,-\delta \,+\,\left(\delta+\frac12\alpha(\pi/2)^2e^{2\delta}+\frac{1}{2\alpha}\right)\vee\left(\frac{\alpha}{2}\eps^2e^{2\max(\delta,0)}\right).
\end{align}
Assume $\eps \ll \frac{\pi}{2}$. For $\delta\geq 0$, we have $$\theta(\delta,\alpha)=\frac12\alpha(\pi/2)^2e^{2\delta}+\frac{1}{2\alpha}, \qquad \inf_{\alpha>0}\theta(\delta,\alpha)=\frac{\pi}{2}e^{\delta}, \qquad \inf_{\alpha>0,\delta>0}\theta(\delta,\alpha)=\frac{\pi}{2}.$$

Let $\delta_*^\eps(\alpha)<0$ be the solution of 
\begin{align}\label{mult:eq:thetaofSEAMdeltcrit}
\delta+\frac12\alpha(\pi/2)^2e^{2\delta}+\frac{1}{2\alpha}=\frac{\alpha}{2}\eps^2.
\end{align}
Note that $\frac12\alpha(\pi/2)^2e^{2\delta}+\frac{1}{2\alpha}$ is atleast $\frac{\pi}{2} e^{\delta}$. And solution of $\delta+\frac{\pi}{2} e^{\delta}=0$ is approximately $-0.745$. So, $\delta_*^0(\alpha)<-0.745$ for any $\alpha$. For $\eps$ very small, $\delta_*^\eps(\alpha)$ will be very close to  $\delta_*^0(\alpha)$.

For $\delta<\delta_*^\eps(\alpha)$, we have
$$\theta(\delta,\alpha)=-\delta+\eps^2\frac{1}{2}\alpha, \qquad \inf_{\delta<\delta_*^\eps(\alpha)}\theta(\delta,\alpha)=-\delta_*^\eps(\alpha)+\eps^2\frac{1}{2}\alpha, \qquad \inf_{\alpha>0}\inf_{\delta<\delta_*^\eps(\alpha)}\theta(\delta,\alpha)=\inf_{\alpha>0}-\delta_*^\eps(\alpha).$$
For $\delta_*^\eps(\alpha)<\delta<0$, we have $\theta(\delta,\alpha)=\frac12\alpha(\pi/2)^2e^{2\delta}+\frac{1}{2\alpha}$, 
$$\inf_{\delta_*^\eps(\alpha)<\delta<0}\theta(\delta,\alpha)=-\delta_*^\eps(\alpha)+\eps^2\frac{1}{2}\alpha, \qquad \inf_{\alpha>0}\inf_{\delta_*^\eps(\alpha)<\delta<0}\theta(\delta,\alpha)=\inf_{\alpha>0}-\delta_*^\eps(\alpha).$$
Because $\delta_*^\eps(\alpha)$ is close to $\delta_*^0(\alpha)$, $\inf_{\alpha>0}-\delta_*^\eps(\alpha)$ would be very close to or greater than $0.745$. 

So, bound given by \cite{SEAM_AOP_2} on $\lambda_1$ is close to $0.745$ but results obtained in this paper indicate that (did not prove) $\lambda_1$ is of order $\eps^2$. The suboptimality of the bounds in \cite{SEAM_AOP_2} for systems satisfying assumption \ref{ass:assumptondetsys} might be because exponential shift by a real number, as done in theorem 4.1 of \cite{SEAM_AOP_2},  does not capture the effect of purely imaginary eigenvalues.
\end{document}